\documentclass[a4paper,11pt]{amsart}

\usepackage[utf8]{inputenc}
\usepackage[english]{babel}
\usepackage{fullpage}
\usepackage[T1]{fontenc}
\usepackage{times}

\usepackage{amsmath}
\usepackage{amsthm}
\usepackage{amsfonts}
\usepackage{amssymb}
\usepackage{amsxtra}
\usepackage{graphicx}
\usepackage{epsfig}
\usepackage{enumerate}
\usepackage{mathrsfs}
\usepackage{tikz}
\usetikzlibrary{arrows}
\usepackage{titletoc}
\usepackage{indentfirst}

\theoremstyle{plain}
\newtheorem{prop}{Proposition}[section]
\newtheorem{theorem}[prop]{Theorem}
\newtheorem{thmx}{Theorem}
\newtheorem{lemma}[prop]{Lemma}
\newtheorem{cor}[prop]{Corollary}
\theoremstyle{definition}
\newtheorem{definit}[prop]{Definition}

\theoremstyle{remark}
\newtheorem{remark}[prop]{\textsc{Remark}}

\newcommand{\jac}{\mathcal{J}ac}

\usepackage{extarrow}
\usepackage{enumerate}
\usepackage{fontenc}
\usepackage[T1]{fontenc}
\usepackage{graphicx}
\usepackage[colorlinks=true, linkcolor=blue]{hyperref}
\usepackage{latexsym}
\usepackage{mathrsfs}
\usepackage{mathtext}
\usepackage{tikz}
\usepackage{textcomp}
\usepackage{upgreek}

\usetikzlibrary{arrows}
\usetikzlibrary{matrix}
\usetikzlibrary{shapes}
\usetikzlibrary{snakes}
\usetikzlibrary{matrix}

\newcommand{\A}{\mathbb{A}}

\newcommand{\E}{\mathcal{E}}

\newcommand{\Gm}{\mathbb G_m}

\newcommand{\N}{\mathbb{N}}

\newcommand{\PP}{\mathbb{P}}

\newcommand{\Sch}{\textup{Sch}}

\newcommand{\Z}{\mathbb{Z}}

\newcommand{\alA}{\mathscr{A}}

\newcommand{\arr}{\longrightarrow}
\newcommand{\arrdi}[1]{\xlongrightarrow{#1}}

\newcommand{\comma}{,\ }

\newcommand{\duale}[1]{{#1}^{\vee}}

\newcommand{\id}{\textup{id}}

\newcommand{\odi}[1]{\mathcal{O}_{#1}}

\newcommand{\sets}{(\textup{sets})}

\newcommand{\shB}{\mathcal{B}}
\newcommand{\shC}{\mathcal{C}}
\newcommand{\shD}{\mathcal{D}}

\newcommand{\shF}{\mathcal{F}}
\newcommand{\shG}{\mathcal{G}}
\newcommand{\shH}{\mathcal{H}}
\newcommand{\shI}{\mathcal{I}}

\newcommand{\shK}{\mathcal{K}}
\newcommand{\shL}{\mathcal{L}}
\newcommand{\shM}{\mathcal{M}}
\newcommand{\shN}{\mathcal{N}}

\newcommand{\shP}{\mathcal{P}}
\newcommand{\shQ}{\mathcal{Q}}
\newcommand{\shR}{\mathcal{R}}
\newcommand{\shS}{\mathcal{S}}
\newcommand{\shT}{\mathcal{T}}
\newcommand{\shU}{\mathcal{U}}

\newcommand{\shW}{\mathcal{W}}

\newcommand{\st}{\ | \ }

\newcommand{\stB}{\mathcal{B}}
\newcommand{\stC}{\mathcal{C}}

\newcommand{\stF}{\mathcal{F}}

\newcommand{\stH}{\mathcal{H}}

\newcommand{\stM}{\mathcal{M}}

\newcommand{\stS}{\mathcal{S}}

\newcommand{\stU}{\mathcal{U}}
\newcommand{\stV}{\mathcal{V}}
\newcommand{\stW}{\mathcal{W}}
\newcommand{\stX}{\mathcal{X}}
\newcommand{\stY}{\mathcal{Y}}
\newcommand{\stZ}{\mathcal{Z}}
\newcommand{\then}{\ \Longrightarrow \ }

\renewcommand{\l}{\textup{l}}

\DeclareMathOperator{\Aut}{\textup{Aut}}
\DeclareMathOperator{\Autsh}{\underline{\textup{Aut}}}
\DeclareMathOperator{\Bi}{\textup{B}}

\DeclareMathOperator{\Coker}{\textup{Coker}}

\DeclareMathOperator{\GL}{\textup{GL}}

\DeclareMathOperator{\Hilb}{\textup{Hilb}}
\DeclareMathOperator{\Hl}{\textup{H}}
\DeclareMathOperator{\Hom}{\textup{Hom}}
\DeclareMathOperator{\Homsh}{\underline{\textup{Hom}}}

\DeclareMathOperator{\Imm}{\textup{Im}}

\DeclareMathOperator{\Iso}{\textup{Iso}}
\DeclareMathOperator{\Isosh}{\underline{\textup{Iso}}}
\DeclareMathOperator{\Ker}{\textup{Ker}}

\DeclareMathOperator{\PGL}{\textup{PGL}}
\DeclareMathOperator{\Pic}{\textup{Pic}}
\DeclareMathOperator{\Picsh}{\underline{\textup{Pic}}}

\DeclareMathOperator{\Proj}{\textup{Proj}}

\DeclareMathOperator{\R}{\textup{R}}

\DeclareMathOperator{\Spec}{\textup{Spec}}

\DeclareMathOperator{\car}{\textup{char}}

\DeclareMathOperator{\pr}{\textup{pr}}

\DeclareMathOperator{\rk}{\textup{rk}}
\DeclareMathOperator{\tr}{\textup{tr}}

\begin{document}

\author[Poma]{Flavia Poma}
\email{flavia.poma@sns.it}
\address[Poma]{}

\author[Talpo]{Mattia Talpo}
\email{talpuz@mpim-bonn.mpg.de}
\address[Talpo]{Max Planck Institute for Mathematics\\Vivatsgasse 7\\
53111 Bonn\\Germany}

\author[Tonini]{Fabio Tonini}
\email{tonini@mathematik.hu-berlin.de}
\address[Tonini]{Humboldt University of Berlin\\Unter den Linden 6\\
10099 Berlin\\Germany}


\title{Stacks of uniform cyclic covers of curves\\
and their Picard groups}
\date{}


\thanks{The second author was supported by the Max Planck Institute for Mathematics of Bonn. The third author was supported by the 
project SFB 647: Space-Time-Matter. Analytic and Geometric Structures.}

\begin{abstract}
We study the stack $\stB_{h,g,n}$ of uniform cyclic covers of degree $n$ between smooth curves of genus $h$ and $g$
and, for $h \gg g$, present it as an open substack
of a vector bundle over the universal Jacobian stack of $\stM_g$. 
We use this description to compute the integral Picard group of $\stB_{h,g,n}$, showing that it is generated 
by tautological classes of $\stB_{h,g,n}$.
\end{abstract}
\maketitle

\setcounter{tocdepth}{1}

\section*{Introduction}
Let $k$ be a field and $h,g,n$ be non negative integers with $n\geq 2$.
We denote by $\stB_{h,g,n}$ the stack over $k$ of triples $(D\arr S,C\arr S,f)$ where
$D\arr S$ is a smooth, geometrically connected genus $h$ curve, $C\arr S$ is a smooth, geometrically connected genus $g$ curve 
and $f\colon D\arr C$ is a uniform cyclic cover of degree $n$
(see Section \ref{sec: uniform covers} for a definition of uniform cyclic covers).
The aim of this paper is to describe the structure of $\stB_{h,g,n}$ and compute its integral Picard group.

This work was inspired by the results in \cite{Arsie2004} and
\cite{Bolognesi2009}, where the authors compute the Picard group of 
similar moduli problems, namely
the stack of uniform cyclic covers of projective spaces and of triple covers of curves of genus zero respectively.
However the methods used here are different,
as we do not use a presentation of $\stB_{h,g,n}$ as a quotient stack.
Another source of inspiration and, in fact, the starting point of our computation in genus one
was the classical result of Mumford about the Picard
group of the stack $\stM_{1,1}$ of elliptic curves (see \cite{Mumford1963} and also \cite{Fulton2010}).

This paper was born as a study of double covers of genus one curves, that is of the stacks $\stB_{h,1,2}$
(which also explains the use of the letter $\stB$ which stands for 'bielliptic').
The main obstacle in generalizing the results for $g\geq 2$ was the computation of the Picard group
of the universal Jacobian of $\stM_g$ (see below for a definition), since the methods we used for the 
same problem in genus one fail in higher genera. This last problem was solved in \cite{melo-viviani},
allowing the generalization for higher genera.

Let us also remark that the case $n=2$ is the most interesting from a ``geometric'' point of view, since (in characteristic different from $2$) all covers of degree $2$ are uniform cyclic and, therefore, $\stB_{h,g,2}$ is the stack of double covers between smooth curves of genera $h$ and $g$.

The Picard group of $\stB_{h,0,n}$ was already computed in \cite[Theorem 5.1]{Arsie2004}.
Here the authors introduce moduli stacks of uniform cyclic covers of projective spaces, denoted by
$\stH_{\text{sm}}(r,n,d)$ for $r,n,d>0$. In the one dimensional case $r=1$ we have $\stH_{\text{sm}}(1,n,d)=\stB_{h,0,n}$,
where $d,n,h$ are related by the expression (\ref{formula for d}) below. In this paper we provide an alternative method
for the computation of $\Pic \stB_{h,0,n}$ which extends to higher genera.

In \cite{Pagani2013} the author introduces moduli stacks of abelian covers of curves, which are
related to our stacks $\stB_{h,g,n}$ in the cyclic, totally ramified case. Let
$\stY_{g,r,n}$ be the stack of tuples $(D,C,f,\sigma_1,\dots,\sigma_r)$
where $(C,\sigma_1,\dots,\sigma_r)$ is a $r$-pointed curve of genus $g$ and $f\colon D\arr C$ is a uniform cyclic cover
of degree $n$ whose ramification locus is the union of the sections $\sigma_1,\dots,\sigma_r$.
By forgetting the sections we obtain a functor $\stY_{g,nd,n}\arr \stB_{h,g,n}$, where $d,g,h,n$ are related by the expression
(\ref{formula for d}) below, which is a $S_{nd}$-torsor.
In \cite[Theorem 3]{Pagani2013} is proved that $\stY_{g,nd,n}$, which is denoted by $\stM_{1,nd}(\Bi (\Z/n\Z),(1),\dots,(1))$, has trivial rational Picard group for $nd>0$, which also implies
the vanishing of the rational Picard group of $\stB_{h,1,n}$. In this paper we recover this last result by
explicitly describing the integral Picard group of $\stB_{h,1,n}$, but we can not directly deduce the result in \cite[Theorem 3]{Pagani2013}.

The main result of this paper is the following.
\begin{thmx}\label{main theorem}
Let $h,g,n$ be non negative integers with $n\geq 2$ and set 
\begin{align}\label{formula for d}
d=2\frac{h+n(1-g)-1}{n(n-1)}\text{ so that } h=1+n(g-1)+\frac{n(n-1)}{2} d
\end{align}
The stack $\stB_{h,g,n}$ is not empty if and only if $d\in \N$. Assume $d\in \N$. The stack $\stB_{h,g,n}$ is algebraic and of finite type over $k$ and, if $nd>2g-2$ or $\car k \nmid n$, the forgetful functor $\stB_{h,g,n}\arr \stM_g$ is smooth and surjective.

Let $\pi\colon \shC\arr \stB_{h,g,n}$ be the universal genus $g$ curve,
$f\colon \shD\arr \shC$ be the universal uniform cyclic cover of
degree $n$ and $\shL$ be the dual of the degree $1$ part of 
the $\mu_n$-equivariant sheaf $f_*\odi \shD$. The sheaf $\shL$ is invertible of degree $d$ over $\shC$ and we have the following.
\begin{enumerate}
 \item \label{main theorem genus zero} If $g=0$ we have
 \begin{displaymath}
  \Pic \stB_{h,0,n} \simeq \begin{cases}
                          \Z/2n(nd-1)\Z &\text{ generated by } \pi_*(\shL \otimes \omega_\pi^{d/2}) \text{ if } d \text{ is even}\\
                          \Z/n(nd-1)\Z &\text{ generated by } \det \pi_*(\shL \otimes \omega_\pi^{(d-1)/2}) \text{ if } d \text{ is odd}
                         \end{cases}
 \end{displaymath}
 \item \label{main theorem genus one big h} If $g=1$ then $\Pic \stB_{h,1,n}$ is generated by $\pi_*\omega_\pi$ and $\det \pi_* \shL$ with relations
 \begin{align*}
(\pi_*\omega_\pi)^4 &\text{ and } \det \pi_*\shL &\text{ if } h=1\comma n=2\comma \car k \nmid 6 \\
(\pi_*\omega_\pi)^6 &\text{ and } (\pi_*\shL)^2\otimes (\pi_*\omega_\pi)^{-2} & \text{ if } h=n=2\comma \car k \neq 2 \\
(\pi_*\omega_\pi)^{12} &\text{ and } (\det \pi_*\shL)^{2n^2}\otimes (\pi_*\omega_\pi)^{n(dn+d-2n)} &\text{ if } nd>2\comma \car k \nmid nd
\end{align*}
As an abtract group we have
      \[
 \Pic \stB_{h,1,n} \simeq \begin{cases}
 				\Z/4\Z & \text{ if } h=1\comma n=2\comma \car k \nmid 3 \\
 				\Z/3\Z \times \Z/2\Z \times \Z/2\Z & \text{ if } h=n=2 \\
 			        \Z/3\Z \times \Z/4\Z\times \Z/2n^2\Z & \text{ if } nd>2\comma \frac{n(dn+d-2n)}{2} \text{ is even}\comma \car k \nmid nd\\
				\Z/3\Z \times \Z/2\Z\times \Z/4n^2\Z & \text{ if } nd>2\comma \frac{n(dn+d-2n)}{2} \text{ is odd}\comma \car k \nmid nd
                          \end{cases}
 \]
 \item \label{main theorem genus big} Assume that $k$ is algebraically closed of characteristic $0$ and either $nd>2g-2$ and $g\geq 4$ or $nd>2g-1$ and $g\geq 3$ or $nd>2g$ and $g\geq 2$. Then $\Pic \stB_{h,g,n}$ is generated by
 $\det \pi_* \omega_\pi$, $d_\pi(\shL)$ and $\det \pi_*(\shL\otimes \omega_\pi)$ (see \ref{notation of picard universal jacobian big genus}
 for a defintion of $d_\pi(-)$) with the only relation
 \[
 (\det \pi_*\omega_\pi)^{-2n^2} \otimes d_\pi(\shL)^{n(n-1)} \otimes (\det \pi_*(\shL\otimes \omega_\pi))^{n(n+1)}
 \]
except for $g=2$, for which we need to add the relation $(\det \pi_* \omega_\pi)^{10}$.

As an abstract group we have
 \[
 \Pic \stB_{h,g,n} \simeq \begin{cases}\Z/2n\Z \times \Z/10\Z \times \Z  & \text{ if } g=2 \text{ and } n \text{ is odd}\\
				    \Z/n\Z \times \Z/10\Z \times \Z  & \text{ if } g=2 \text{ and } n \text{ is even}\\
				    \Z/2n\Z \times \Z^2  & \text{ if } g>2 \text{ and } n \text{ is odd}\\
				    \Z/n\Z \times \Z^2  & \text{ if } g>2 \text{ and } n \text{ is even}
                                   \end{cases}
 \]
\end{enumerate}
\end{thmx}

When $h=n(g-1)+1$ (in particular $g\geq 1$),
that is $d=0$, uniform cyclic covers of degree $n$ become $\mu_n$-torsors. This case is not covered by the above theorem, except for
 $h=g=1$ and $n=2$ where $\Pic \stB_{1,1,2}$ is computed by an ad hoc variation of the methods used in higher genera and degrees.
 When $d>0$ uniform cyclic covers are never \'etale and the stacks $\stB_{h,g,n}$ share a common description that we are now going to explain.
 
 Fix $g\geq 0$, $d,n>0$ and assume $\car k \nmid n$ or $g=0$ and, if $g=1$, $\car k \nmid d$. 
 Set $\stU_{d,g,n}$ for the stack of triples $(C\arr S,\shQ,s)$ where $C$ is a smooth, geometrically connected curve of genus $g$ and $\shQ$ 
is a degree $d$ invertible sheaf with a section $s\in \shQ^n$ that is not identically zero on any of the geometric fibers of $C\arr S$.
The forgetful map $\stU_{d,g,n}\arr \stM_g$ defines the (universal) genus $g$ curve $\pi\colon \shC \arr \stU_{d,g,n}$ together with
an invertible sheaf $\shL$ of degree $d$ on $\shC$ and a section $t\in \shL^n$. The zero locus $\stH_{d,g,n}$ of $t$ in $\shC$ is a degree $nd$ cover
of $\stU_{d,g,n}$ and $\stB_{h,g,n}$, where $d,h,g,n$ are related by the expression (\ref{formula for d}) above, can be identified with 
the \'etale locus of $\stH_{d,g,n}\arr \stU_{d,g,n}$ inside $\stU_{d,g,n}$ (see \ref{relation between Bhg and Udtwo}).
Since $\stU_{d,g,n}$ is smooth and algebraic, the problem of computing
$\Pic \stB_{h,g,n}$ splits in two parts: compute $\Pic \stU_{d,g,n}$ and describe the complement of $\stB_{h,g,n}$ in $\stU_{d,g,n}$.

Denote by $\jac_{d,g}$ the stack of pairs $(C\arr S,\shQ)$ where $C$ is a smooth, geometrically connected 
genus $g$ curve and $\shQ$ is a degree $d$ invertible sheaf on $C$.
This is the so called universal Jacobian of degree $d$ on $\stM_g$. Assume $nd>2g-2$. The forgetful functor $\stU_{d,g,n}\arr \jac_{d,g}$ makes $\stU_{d,g,n}$ into the complement of the zero section of a vector bundle of rank
$nd+1-g$ over $\jac_{d,g}$ (this description is no longer true in general when $nd\leq 2g-2$ and this is why Theorem
\ref{main theorem} does not cover this case). When $n\geq 2$ and since $nd+1-g\geq 2$ we can conclude that $\Pic \stU_{d,g,n}\simeq \Pic \jac_{d,g}$.
When $g\geq 2$ the group $\Pic \jac_{d,g}$ has been computed in \cite{melo-viviani}. If $g=0$ then $\Pic \jac_{d,0} \simeq \Z$ (see \ref{picard jacobian genus zero}).
The case $g=1$ is harder than the case $g=0$ and our treatment differs from the methods used in \cite{melo-viviani} for $g\geq 2$.
The result is that, if $g=1$, then $\Pic \jac_{d,1}\simeq \Z/12 \times \Z$ (see \ref{picard group of Xd}) and it has been obtained by
proving that the functor $\jac_{d,1} \arr \stM_{1,1}$ that maps $(E,\shQ)$ to $(\Picsh^0_E,[\odi E])$ is a trivial gerbe, that is
$\jac_{d,1} \simeq \Bi_{\stM_{1,1}}G_d$ over $\stM_{1,1}$, where $G_d$ is a smooth affine group over $\stM_{1,1}$, and by computing the group
of characters of $G_d$. The geometric fibers
of $G_d\arr \stM_{1,1}$ are particular cases of Theta groups,
first defined by Mumford in his paper \cite{Mumford1966} (see \ref{properties of Gd}).

The last part in the computation of $\Pic \stB_{h,g,n}$ is the description of $\stU_{d,g,n}-\stB_{h,g,n}$. Topologically
this closed substack coincides with the discriminant locus $\stZ_{d,g,n}$ of the cover $\stH_{d,g,n}\arr \stU_{d,g,n}$. By standard theory of covers
the locus $\stZ_{d,g,n}$ can be described as the zero locus of a section (the discriminant section) of an invertible sheaf of $\stU_{d,g,n}$. The key
point for the computation of $\Pic \stB_{h,g,n}$ is that $\stZ_{d,g,n}$ is reduced and, if $nd>2g$ or $nd>2g-1$ and $g\geq 3$ or $nd>2g-2$ and $g\geq 4$, irreducible (see \ref{key proposition for the bad locus}).
In the special case $g=1$, $h=n=2$ (so that
$d=1$ and $nd=2g$) ad hoc methods show that $\stZ_{1,1,2}$ is a disjoint union of two irreducible components, allowing the computation of
$\Pic \stB_{2,1,2}$ (see \ref{key proposition for the bad locus}). It is not clear whether $\stZ_{d,g,n}$ is irreducible for the remaining values of $h,g,n$. The geometry of the loci $\stZ_{d,g,n}$ is studied by reducing to the case $n=1$, showing that $\stU_{d,g,1}\simeq \Hilb^d_{\stM_{g,1}/\stM_g}$ 
(see section Notations) and that $\stH_{d,g,1}\arr \Hilb^d_{\stM_{g,1}/\stM_g}$ is the universal cover.

We remark that Theorem \ref{main theorem} is obtained by expressing $\Pic \stB_{h,g,n}$ as quotient of $\Pic \jac_{d,g}$ by a given relation and this description holds more generally than in the hypothesis of Theorem \ref{main theorem} (see \ref{general prop for main theorem} for a precise statement). For instance if $g\geq 2$ and $\car k \nmid n$ the knowledge of the integral (resp. rational) Picard group of $\jac_{d,g}$ implies the knowledge of the integral (resp. rational) Picard group of $\stB_{h,g,n}$. Unfortunately if $g\geq 2$ both the integral and rational Picard groups of $\jac_{d,g}$ are known only in characteristic $0$, although it seems reasonable to expect the same description for all but finitely many characteristics. See \cite[Remark 1.4]{melo-viviani} for a discussion on the subject. 

The paper is organized as follows.
In Section \ref{Preliminaries} we collect useful remarks and lemmas, while in Section \ref{Stack Xd}
we study the Picard group of the univeral Jacobian $\jac_{d,g}$ over $\stM_g$ for $g=0$ and $g=1$ and explain
the results for $g\geq 2$ obtained in \cite{melo-viviani}.
In Section \ref{Canonical covers and their discriminant loci} we introduce the
canonical covers $\stH_{d,g,n}\arr \stU_{d,g,n}$ and describe their discriminant loci, while in Section 
\ref{sec: uniform covers} we introduce the stacks $\stB_{h,g,n}$ and compute their Picard groups.

\subsection*{Notations}
Given a base scheme $S$, by the words ``scheme'' or ``stack'' we always mean scheme or stack defined over this base scheme.
Moreover by stacks we always mean a category fibered in groupoids which is a stack for the fppf topology.
Let $\stS$ be a stack.

A geometric point of $\stS$ is a map $\Spec k\arr \stS$, where $k$ is an algebraically closed field.

Given a sheaf of groups $G\colon (\Sch/\stS)^{\text{op}}\arr (\text{grps})$ we denote by $\duale G=\Hom_\stS(G,\Gm)$ the group of characters of $G$
over $\stS$, i.e. group homomorphisms $G\arr \Gm$.

By a cover we mean an affine map $\stX\arrdi f \stS$ such that $f_*\odi\stX$ is locally free of finite rank. Alternatively, a cover is
a finite, flat and finitely presented map $\stX\arrdi f \stS$. The degree of $f$ is the rank of $f_*\odi \stX$.
 The discriminant section $s_f\in \det (f_*\odi \stX)^{-2}$ of $f$ is the determinant of the map
$$
 f_*\odi \stX \arr \duale {f_*\odi \stX}\comma x\longmapsto \tr_{f_*\odi \stX}(x\cdot -)
$$
where $\tr$ denotes the trace map. The discriminant locus of $f$ is the zero locus in $\stS$ of $s_f$.
The discriminant section is stable under base change and the complement of the discriminant locus of $f$ is the \'etale locus of $f$ in $\stS$.

A genus $g$ curve over $\stS$ is a representable (by algebraic spaces), proper and smooth map $\stC\arr\stS$ of stacks
whose geometric fibers are connected genus $g$ curves. Let $\pi\colon \stC\arr \stS$ be a genus $g$ curve.
We say that an invertible sheaf $\shL$ on $\shC$ has degree $d\in \Z$ if the pullback of $\shL$ on every geometric fiber 
of $\pi\colon \stC\arr \stS$ has degree $d$. We denote by
$\omega_\pi$ the relative dualizing sheaf, which is an invertible sheaf on $\shC$ of degree $2g-2$.

We denote by $\stM_g$ the stack
of genus $g$ curves and by $\stM_{g,1}$ the stack of genus $g$ curves with a section.
The forgetful functor $\stM_{g,1}\arr \stM_g$ is a genus $g$ curve, called the universal genus $g$ curve of $\stM_g$.
If it is given a map $\stS\arr\stM_g$, the universal genus $g$ curve of $\stS$ is the base change $\stS\times_{\stM_g}\stM_{g,1}\arr\stS$.

Let $\stX$ be another stack and $f\colon \stX\arr\stS$ be a representable map.
We denote by $\Hilb^{n}_{\stX/\stS}$ (or simply $\Hilb^{n}_\stX$ whenever $\stS$ is clear from the context) the stack over $\stS$ whose
objects over $S\arr\stS$ are closed subschemes $Z \subseteq \stX\times_\stS S$ such that the projection $Z\arr S$ is a degree $n$ cover. When $f\colon \stX\arr \stS$ is a projective map of schemes then $\Hilb^n_{\stX/\stS}$ is the usual Hilbert scheme of points.


We denote by $\Picsh_{\stX/\stS}$ (or simply $\Picsh_\stX$ whenever $\stS$ is clear from the context) the stack over $\stS$ which is the 
fppf sheafification of the functor $(\Sch/\stS)^{\text{op}}\arr \sets$ that maps $S\arr\stS$ to the set $\Pic(\stX\times_\stS S)$.
If $f\colon \stX\arr\stS$
is a curve, we also denote by $\Picsh^d_{\stX/\stS}$ (or simply $\Picsh^d_\stX$) the substack of $\Picsh_{\stX/\stS}$ of classes that
are locally given by invertible sheaves having degree $d$ on the geometric fibers of $f$. Given $n\in \Z$ we will denote by $[n]\colon \Picsh_{\stX/\stS} \arr \Picsh_{\stX/\stS}$ (or $[n]\colon \Picsh^d_{\stX/\stS}\arr \Picsh^{dn}_{\stX/\stS}$ if $\stX$ is a curve over $\stS$) the map induced by the multiplication by $n$.

The formation of $\Hilb^{n}_{\stX/\stS}$, $\Picsh_{\stX/\stS}$ and $\Picsh^d_{\stX/\stS}$ commute with arbitrary
base change of the base $\stS$. In particular, if $\stX\arr\stS$ is a curve, the stacks $\Hilb^d_{\stX/\stS}$ and $\Picsh^d_{\stX/\stS}$ for
$d\in \Z$ are smooth over $\stS$.

If $\stX$ is an algebraic stack we will denote by $|\stX|$ the topological space associated with $\stX$.

Almost all the stacks $\stX$ that we will introduce have a given map to $\stM_g$, for some given $g$,
and therefore have a given genus $g$ curve over it, that, as
remarked above, we will call the universal curve over (or of) $\stX$. With abuse of notation
we will usually denote them by the same symbol $\shC$ for the total space and $\pi$ for the structure map, that is $\pi\colon \shC\arr\stX$,
but with the convention that this notation is fixed and remains coherent inside the statement of a lemma, proposition, theorem ... and its proof.
The use of different symbols for such curves seemed to us not practical, while the use of subscripts
would have encumbered the notation too much. Moreover this notation is supported by the idea that genus $g$ curves can be seen as restriction
of the universal curve over $\stM_g$. Indeed, 
if $\pi\colon \shC\arr\stM_g$ is the universal curve and $q\colon \stX\arr\stM_g$ is a map, $\shC$ can be seen as
the functor $F\colon (\Sch/\stM_g)^{\text{op}}\arr\sets$ which maps a genus $g$ curve $C$ over a scheme $S$ to the 
set of sections $C(S)$, while $q$ corresponds to a map $\Sch/\stX\arr \Sch/\stM_g$. The universal curve
$\shC\times_{\stM_g} \stX\arr\stX$ of $\stX$ then corresponds to the restriction of the functor $F$ along the map $\Sch/\stX\arr\Sch/\stM_g$.

\subsection*{Acknowledgements}
We would like to thank our advisor Angelo Vistoli for suggesting the problem and for sharing his ideas.

We also thank Michele Bolognesi for directing us to Mumford's papers about abelian varieties,
along with Dajano Tossici, Matthieu Romagny, Nicola Pagani, Filippo Viviani and Margarida Melo for useful conversations.

Finally we thank the referee for useful comments and corrections.




\section{Preliminaries}\label{Preliminaries}

In this section we collect some general results and remarks that will be useful in the next sections.
These results are well known, but for some of them we decided to include a proof for
completeness and lack of exhaustive references.
In this section we consider $\Spec \Z$ as the base scheme.

\begin{prop}\label{cokernel of a geometrically injective map of flat sheaves}
 Let $f\colon \stX\arr \stY$ be a locally finitely presented map of algebraic stacks and $\alpha\colon \shF\arr \shH$ be a map of finitely presented quasi-coherent sheaves on $\stX$. If $\shH$ is flat over $\stY$ then $\alpha$ is injective on the geometric fibers of $f$ if and only if $\alpha$ is injective and $\Coker \alpha$ is flat over $\stY$. In this case $\alpha$ remains injective after any base change $\stY'\arr \stY$ from an algebraic stack.
\end{prop}
\begin{proof}
 It is easy to see that taking atlases of $\stY$ and $\stX$ we can reduce the problem to the case of schemes, where it follows from \cite[Proposition 11.3.7]{EGAIV-3}.
\end{proof}

\begin{remark}\label{picard of BG}
 Let $\stY$ be an algebraic stack and $G$ be a sheaf of groups over $\stY$. Then we have a natural isomorphism
$$
  \Pic \Bi_\stY G \simeq \Pic \stY \oplus \Hom(G,\Gm)
$$
Indeed, by descent, an invertible sheaf over $\Bi_\stY G$ is given by a pair $(\shL,\rho)$ where $\shL$ is an invertible sheaf over $\stY$ and
$\rho\in \Hom(G,\Gm)$, which defines an action of $G$ on $\shL$ given by $\rho\colon G\arr \Gm\simeq \Autsh(\shL)$.
\end{remark}

In some proofs we will use dimension counting for algebraic stacks. We recall here some properties which are well known for schemes.
We refer to \cite[Chapter 11]{Laumon1999} for definitions and basics about dimension theory for stacks.

\begin{remark}
 Let $f\colon \stX\arr \stY$ be a locally of finite type map of algebraic stacks. If $\xi\in |\stX|$, $\eta \colon \Spec k\arr \stY$, where $k$ is a field, maps to $f(\xi)$  and $x \in |\stX\times_\stY k|$ maps to $\xi$ via the projection then the number
  \begin{displaymath}
   \dim_\xi f = \dim_x (\stX\times_\stY k)\in \Z
  \end{displaymath}
  does not depend on $\eta$ and $x$. Indeed by standard arguments about fiber products one can reduce to the case $\stY=\Spec k$ and show that if $L/k$ is a field extension and $\xi'\in |\stX\times_k L|$ maps to $\xi\in |\stX|$ then $\dim_{\xi'}(\stX\times_k L)=\dim_\xi \stX$. Using the definition of dimension for stacks one can assume that $\stX$ is a scheme. In this case the result is standard (see for instance \cite[Tag 02FW]{SP014}).
\end{remark}

\begin{definit}
 Given $n\in \Z$, a locally of finite type map $f\colon \stX\arr \stY$ of algebraic stacks has (pure) \emph{relative dimension} $n$ if all (the irreducible components of all) the fibers have dimension $n$.
\end{definit}

\begin{remark}\label{equidimensional atlases}
 If $\stX$ is a quasi-compact algebraic stack there exists $n\in \N$ and an atlas $X\arr \stX$ of pure relative dimension $n$, where $X$ is a quasi-compact scheme. Indeed if $P\colon X'\arr \stX$ is an atlas from a quasi-compact scheme $X'$, by \cite[Proposition 11.10]{Laumon1999} we have a decomposition $X'=\bigsqcup_{r=0}^n X_r$ such that, if $x\in X_r$, $r=\dim_x P(=\dim_x P_{|X_r})$. Taking into account \cite[Corollary 11.11]{Laumon1999} the map $Q\colon X=\bigsqcup_{r=0}^n \A^{n-r}_{X_r}\arr \stX$ satisfies $\dim_x Q=n$ for all $x\in X$, i.e. it has pure relative dimension $n$.
\end{remark}

\begin{prop}\label{dimension formulas for stacks}
Let $f\colon \stX\arr \stY$ be a flat and locally of finite type map of locally noetherian algebraic stacks. Then
\begin{align}\label{formula for dimension of stacks}
 \dim_\xi \stX = \dim_\xi f + \dim_{f(\xi)} \stY \text{ for all }\xi\in |\stX|
\end{align}
In particular if $f$ has relative dimension $r\in \Z$ then
\begin{displaymath}
 \dim \stX=\dim \stY + r
\end{displaymath}
Moreover if $\stY$ is locally of finite type over a field or $\Z$, $f$ has pure relative dimension $r\in \Z$ and $\stY'$ is an irreducible component of $\stY$ then all irreducible components of $f^{-1}(\stY')$ dominate $\stY'$ and have dimension $\dim \stY'+r$.
\end{prop}
\begin{proof}
We first prove (\ref{formula for dimension of stacks}) when $\stX$ and $\stY$ are schemes. By \cite[Corollary 14.2.6]{EGAIV-3} and since fibers have the subspace topology one can assume $f$ surjective, of relative dimension $n\in \N$ and translate (\ref{formula for dimension of stacks}) in $\dim \stX=\dim \stY+n$. By \cite[Theorem 14.2.1]{EGAIV-3} we have $\dim (\odi{\stX,x})=\dim (\odi{f^{-1}(f(x)),x})+\dim (\odi{\stY,f(x)})$ for all $x\in \stX$. Since for all $y\in \stY$ we have $\max_{x\in f^{-1}(y)}\{\dim (\odi{f^{-1}(y),x})\}=n$ we get the desired expression.

 When $f$ is a smooth atlas (\ref{formula for dimension of stacks}) follows from definition of $\dim_{f(\xi)}\stY$. We show that 
 \begin{displaymath}
  (\ref{formula for dimension of stacks})\text{ for schemes}\then (\ref{formula for dimension of stacks})\text{ for algebraic spaces} \then (\ref{formula for dimension of stacks})\text{ for stacks}
 \end{displaymath}
Both implications follows from the same proof. One considers smooth atlases $X\arr \stX$ and $Y\arr \stY$, choose a point $y\in Y$ over $f(\xi)$ and $x\in (X\times_\stY k(y))$ mapping to $\xi$. Writing a diagram of all possible fiber products the proof now consists in applying (\ref{formula for dimension of stacks}) several times on various projections of this diagram.
 
 If $f$ has relative dimension $r$ the expression $\dim \stX=\dim \stY + r$ follows from (\ref{formula for dimension of stacks}) and the fact that for all $\eta \in |\stY|$ we have $\max_{\xi\in f^{-1}(\eta)}\{\dim_\xi f\}=r$. We now consider the last claim in the statement. We can assume $\stY$ irreducible. Let $\stX'$ be an irreducible component of $\stX$ with generic point $\xi$. Notice that $\stX'$ contains an open substack of $\stX$. In particular $\dim_\xi \stX=\dim_\xi \stX'$. Moreover since $f$ is open $f(\xi)=\eta$ is the generic point of $\stY$. By (\ref{formula for dimension of stacks}) we have $\dim_\xi \stX'= \dim_\eta \stY + r$. Thus it suffices to show that if $\stZ$ is an irreducible stack locally of finite type over a field or $\Z$ and $\zeta$ is its generic point then $\dim_\zeta \stZ=\dim \stZ$. First we can assume $\stZ$ quasi-compact by taking an open substack of $\stZ$ of the same dimension. Then by \ref{equidimensional atlases} there exists an atlas $P\colon Z\arr \stZ$ of pure 
relative dimension $r$, so that $\dim Z=\dim \stZ+r$. On the other hand we have seen that if $z$ is a generic point of $Z$ then $\dim_z Z=\dim_\zeta \stZ+r$. This tells us that $\dim_z Z$ does not depend on the choice of the generic point and therefore, by \cite[Corollay 10.6.4, Example 10.7.1]{EGAIV-3}, $\dim_z Z=\dim \overline{\{z\}}=\dim Z$, which implies $\dim \stZ=\dim_\zeta \stZ$.
\end{proof}

\begin{cor}\label{dimension of Cartier divisors for stacks}
 Let $\stX$ be an irreducible stack of finite type over a field and $\stZ$ be the zero locus of a section of an invertible sheaf on $\stX$. If $\emptyset\neq \stZ\subsetneq \stX$ then all irreducible components of $\stZ$ have dimension $\dim\stX - 1$.
\end{cor}
\begin{proof}
 Let $P\colon X\arr \stX$ be an atlas of pure relative dimension $r$ (see \ref{equidimensional atlases}), $\stZ'$, $Z'$ and $X'$ be irreducible components of $\stZ$, $P^{-1}(\stZ')$ and $X$ such that $Z'\subseteq X'$. Notice that $Z'\subsetneq X'$ because otherwise $P(Z')$ contains the generic point of $\stX$. Since $Z'$ is an irreducible component of a section of an invertible sheaf on $X'$ we have $\dim Z'=\dim X'-1$. On the other hand since $P$ has pure relative dimension $r$ we have $\dim Z'=\dim \stZ'+r$ and $\dim X'=\dim \stX+r$.
\end{proof}

\begin{cor}\label{dimension of image for stacks}
 Let $f\colon \stX\arr\stY$ be a map of stacks locally of finite type over a field and assume that $\stX$ is a Deligne-Mumford stack. Then $\dim \overline{f(\stX)}\leq \dim \stX$, where $\overline{f(\stX)}$ is the reduced closed substack of $\stY$ whose topological space is $\overline{f(|\stX|)}$.
\end{cor}
\begin{proof}
 When $\stX$ and $\stY$ are schemes the result is standard. We show how to reduce to this case. We can assume that $f$ is dominant so that $\overline{f(\stX)}=\stY$. By taking an atlas of $\stY$ of pure relative dimension (see \ref{equidimensional atlases}) we can assume that $\stY$ is a scheme. Moreover we can replace $\stX$ by a scheme because \'etale atlases do not change dimension.
\end{proof}

\begin{prop}\label{relations picard group smooth stacks}
 Let $\stX$ be a smooth and integral algebraic stack over a field.
 \begin{itemize}
 \item If $h\colon \stV\arr\stX$ is a vector bundle of finite rank then $h^*\colon \Pic \stX\arr\Pic \stV$ is an isomorphism.
 \item If $\stZ$ is a closed substack of $\stX$ of codimension greater than $2$ then the restriction map $\Pic \stX \arr \Pic (\stX-\stZ)$ is an isomorphism.
 \item Given $\shL_1,s_1,\dots,\shL_r,s_r$, where $\shL_i$ is an invertible sheaf on $\stX$ with a non zero global section $s_i$ whose zero locus $Z(s_i)$ is integral, then the restriction map induces an isomorphism
 \begin{displaymath}
  \Pic \stX/\langle \shL_1,\dots,\shL_r\rangle \simeq \Pic (\stX-(Z(s_1)\cup \cdots \cup Z(S_r)))
 \end{displaymath}
 \end{itemize}
\end{prop}
\begin{proof}
 Let $\sigma\colon \stX\arr \stV$ be the zero section. We must prove that $\Pic\stV\arrdi{\sigma^*} \Pic\stX$ is injective. Let $\shQ$ be an invertible sheaf on $\stV$ in the kernel and define the sheaf $F_\shQ$ on the small smooth-\'etale site of $\stX$ by
 \begin{displaymath}
  F_\shQ(U)=\Iso_{h^{-1}(U)}(\shQ_{|h^{-1}(U)},\odi{h^{-1}(U)})
 \end{displaymath}
The map $\sigma$ induces a map $F_\shQ\arr \Isosh_\stX (\sigma^*\shQ,\odi \stX)$ and it suffices to prove that it is an isomorphism. This is a local question, so that we can assume $\stX=\Spec D$ affine and $\stV$ trivial. By standard intersection theory for schemes $\shQ$ is trivial and therefore the previous map on the global sections is just $(D[x_1,\dots,x_n])^* \simeq D^*$.

Let $\stU$ be an open substack of $\stX$. If $\shQ$ is an invertible sheaf on $\stU$ then by \cite[Corollary 15.5]{Laumon1999} there exists a coherent sheaf $\shF$ on $\stX$ such that $\shF_{|\stU}\simeq \shQ$. Then the sheaf $\shL=\shF^{\vee\vee}$ is a reflexive sheaf of rank $1$ and thus invertible and $\shL_{|\stU}\simeq \shQ$. This shows that $\Pic \stX\arr \Pic \stU$ is surjective. 
We now use the description of divisors given in \cite[Proof of Lemma 5.2]{Arsie2004}. 
Let $\shL$ be an invertible sheaf on $\stX$ such that $\shL_{|\stU}\simeq \odi{\stU}$.  
It follows that there is a divisor $\shD$ on $\stX$ such that $\shL\simeq \odi \stX(\shD)$ and the support of $\shD$ is in $\stX-\stU$. 
In particular if $\stX-\stU$ has codimension greater than $2$ then $\shD=0$. For the last point the sheaves $\shL_i \simeq \odi \stX(Z(s_i))$ restrict to $\odi \stU$ on $\stU$ and $\shL\simeq \shL_1^{m_1}\otimes \shL_r^{m_r}$ where $m_i$ is the multiplicity of $\shD$ in $Z(s_i)$ (computed on an atlas).
\end{proof}


\begin{lemma}\label{cbs for elliptic curves}
Let $\pi\colon\shC\arr\stS$ be a genus $g$ curve over an algebraic stack and $\shF$ be a finitely presented quasi-coherent sheaf on $\shC$,
flat over $\stS$. Then $\R^j\pi_*\shF$ is locally free and satisfies base change for all $j\in \N$ in the following cases.
\begin{enumerate}
 \item $\shF$ is an invertible sheaf on $\shC$ of degree $0$ such that $[\shF]=0$ in $\Picsh^0_{\shC/\stS}$. In this case $\pi_*\shF$ is an invertible sheaf and
 the canonical map
\[
\pi^*\pi_*\shF\arr\shF
\]
is an isomorphism.
\item $\shF=\omega_\pi$. The sheaf $\pi_*\omega_\pi$ has rank
$g$ and $\R^1\pi_*\omega_\pi\simeq \odi \stS$. Moreover if $g=1$ the map $\pi^*\pi_*\omega_\pi\simeq\omega_\pi$ is an isomorphism.

\item $\shF$ is an invertible sheaf on $\shC$ of degree $d>2g-2$ or $d<0$. In this case $\rk \pi_*\shF=\max\{d+1-g,0\}$
and $\rk \R^1\pi_*\shF=\max\{-d-1+g,0\}$.
\item $\shF$ is supported on a closed substack of $\shC$ which is quasi-finite over $\stS$. In this case $\R^1\pi_*\shF=0$.
\end{enumerate}
In all of the above cases but the last one we have an isomorphism
\[
\R^1\pi_*\shF\simeq \duale{\pi_*(\duale \shF \otimes \omega_\pi)}
\]
\end{lemma}
\begin{proof}
By \cite[II, Definition 10 and Theorem 21]{Kleiman1980} there is a canonical map 
\[
 \pi_*\Homsh(\shF,\omega_\pi)\arr \duale{(\R^1\pi_*\shF)}
 \]
 which is an isomorphism if $\R^1\pi_*\shF$ satisfies base change. In this case if $\shF$ and $\R^1\pi_*\shF$ are locally free we get the last formula in the statement by dualizing the above isomorphism.
 
 All the other claims follow by standard semicontinuity theorems and Riemann-Roch.
\end{proof}

\begin{lemma}\label{degree one sheaves and sections}
 Let $\pi\colon \shC \arr \stS$ be a genus $g$ curve over an algebraic stack and $\shQ$ be a degree $d$ invertible sheaf on $\shC$ with a section $s\in \shQ$ which is non zero on the geometric fibers. Then the zero locus $\stZ$ of $s$ in $\shC$ is a degree $d$ cover of $\stS$. When $d=1$ this defines a section $\tau\colon \stS\arr \shC$ with an isomorphism $\odi \shC(\tau)\simeq \shQ$ sending $1$ to $s$. If in addition $g=1$ then the map $\odi \stS\arrdi s \pi_*\shQ$ is an isomorphism.
\end{lemma}
\begin{proof}
 By \ref{cokernel of a geometrically injective map of flat sheaves} the sequence
 \[
 0\arr \shQ^{-1}\arr \odi \shC \arr \odi \stZ\arr 0
 \]
 is universally exact over $\stS$ and $\stZ$ is flat over $\stS$. Moreover $\stZ\arr\stS$ is proper, finitely presented and, by looking at the geometric
 fibers, quasi-finite. By \cite[Theorem 8.11.1]{EGAIV-3} 
 we can conclude that $\stZ\arr\stS$ is a cover. By Riemann-Roch it has degree $d$.
 
 Assume $d=1$. The claim about the section $\tau$ follows from standard arguments.
 The last claim follows from the fact that $\pi_*\shQ$ is invertible and satisfies base change by \ref{cbs for elliptic curves} and $s\in \pi_*\shQ$
 is nowhere vanishing by hypothesis.
\end{proof}

\begin{remark}\label{multiplication by two on jacobians}
 Let $\pi\colon \shC\arr \stS$ be a genus $g$ curve over an algebraic stack and $n\in \Z$.
 The map $[n]\colon \Picsh^r_{\stC/\stS} \arr \Picsh^{rn}_{\stC/\stS}$ is a cover of degree $n^{2g}$ and it is \'etale
 if $n\in \odi \stS^*$.
 Indeed, since the problem is local on $\stS$, one can assume that $\stS$ is a noetherian scheme and that $\shC\arr \stS$ has a section. This allows
 to reduce the problem to the case $r=0$. Since $\Picsh^0_{\stC/\stS}\arr \stS$ is flat and proper of relative dimension $g$, by the local flatness criterion \cite[Theorem 11.3.10]{EGAIV-3} we
 can assume that $\stS$ is the spectrum of an algebraically closed field. In this case the result follows
 from \cite[Proposition 7.1 and Theorem 7.2]{Milne2008}.
 
 In particular all invertible sheaves on $\shC$ of degree divisible by $n$ are fppf locally (on $\stS$) a $n$-th power of an invertible sheaf on $\shC$.
 Moreover $(\Picsh^0_{\stC/\stS})[n]$ is a finite, flat and 
 finitely presented group scheme over $\stS$ of degree $n^{2g}$ and it is \'etale if $n\in \odi \stS^*$.
\end{remark}

\section{The universal jacobian of degree $d$ over $\stM_g$ and its Picard group.}\label{Stack Xd}

In this section we assume to work over a field of characteristic $p\geq 0$ and we fix a non negative integer $g$ (the genus) and an integer $d$ (the degree).

\begin{definit}\label{definition of Xd}
  We denote by $\jac_{d,g}$ the stack of pairs $(C,\shQ)$
  where $C$ is a curve of genus $g$ and $\shQ$
  is an invertible sheaf over $C$ of degree $d$. The stack $\jac_{d,g}$ is called the \emph{universal Jacobian} stack of degree $d$ over $\shM_g$.
\end{definit}

The aim of this section is to describe the Picard group of $\jac_{d,g}$. When $g\geq 2$ this has already been done in \cite{melo-viviani}. We will deal with the remaining cases, that is genus zero and one.

\begin{remark}\label{smoothness of the universal Jacobian}
The forgetful functor $\jac_{d,g}\arr \stM_g$ is the composition of a $\Gm$-gerbe $\jac_{d,g} \arr \Picsh^d_{\stM_{g,1}/\stM_g}$ and
the representable and smooth functor $\Picsh^d_{\stM_{g,1}/\stM_g}\arr \stM_g$.
In particular $\jac_{d,g}$ is a smooth and integral algebraic stack.
\end{remark}

\begin{definit}
 Let $\shC\arr \jac_{d,g}$ be the universal curve over $\jac_{d,g}$. By construction, there exists an invertible sheaf $\shL$ over $\shC$ such that
    \[
  \begin{tikzpicture}[xscale=1.6,yscale=-0.5]
    \node (A0_0) at (0, 0) {$C$};
    \node (A0_1) at (1, 0) {$\shC$};
    \node (A1_3) at (3, 1) {$q^*\shL\simeq\shQ$};
    \node (A2_0) at (0, 2) {$T$};
    \node (A2_1) at (1, 2) {$\jac_{d,g}$};
    \path (A0_0) edge [->]node [auto] {$\scriptstyle{q}$} (A0_1);
    \path (A0_1) edge [->]node [auto] {$\scriptstyle{}$} (A2_1);
    \path (A0_0) edge [->]node [auto] {$\scriptstyle{}$} (A2_0);
    \path (A2_0) edge [->]node [auto] {$\scriptstyle{(C,\shQ)}$} (A2_1);
  \end{tikzpicture}
  \]
for all schemes $T$. We call the sheaf $\shL$ the \emph{universal invertible sheaf} over $\shC$. Given a stack $\stY$ over $\jac_{d,g}$ the universal invertible
sheaf over the universal curve  $\shC\times_{\jac_{d,g}}\stY$ of $\stY$ is the pull-back of $\shL$ via the map $\shC\times_{\jac_{d,g}} \stY\arr \shC$.
\end{definit}

We now describe the result in \cite{melo-viviani} about $\Pic\jac_{d,g}$ when $g\geq 2$.

\begin{remark}\label{notation of picard universal jacobian big genus}
 Let $\pi\colon \shC\arr \shS$ be a genus $g$ curve. Given an invertible sheaf $\shT$ on $\shC$ one can define an invertible sheaf $d_\pi(\shT)$ on $\shS$, called the determinant of cohomology of $\shT$. When $\pi_*\shT$ and $\R^1\pi_*\shT$ are locally free one can simply set $d_\pi(\shT)\simeq \det \pi_*\shT \otimes (\det\R^1\pi_*\shT)^{-1}$. We refer to \cite{melo-viviani} for the general definition. In this paper we just use the fact that the formation of $d_\pi(\shT)$ commutes with arbitrary base changes.
 
 Notice that from \ref{cbs for elliptic curves} it follows that $d_\pi(\omega_\pi)\simeq \det \pi_*\omega_\pi$ and that, if $\shT$ is an invertible sheaf on $\shC$ of positive degree, then $d_\pi(\shT\otimes \omega_\pi)\simeq \det \pi_*(\shT\otimes \omega_\pi)$.
\end{remark}

\begin{theorem}\label{picard universal jacobian big genus}\cite{melo-viviani}
Assume that the ground field is algebraically closed of characteristic $0$ and that $g\geq 2$ and $d>0$. 
 Let $\pi\colon \shC\arr \jac_{d,g}$ be the universal curve and $\shL$ be the universal invertible sheaf
 over $\shC$. Then $\Pic \jac_{d,g}$ is freely generated by $\det \pi_* \omega_\pi$, $d_\pi(\shL)$ and
 $\det \pi_*(\shL \otimes \omega_\pi)$, except for $g=2$, in which case there is a single relation given by $(\det \pi_*\omega_\pi)^{10}$. Moreover for all $n,k \geq 1$ we have an isomorphism
 \begin{displaymath}
  \det \pi_*(\shL^n \otimes \omega_\pi^k)\simeq (\det \pi_*\omega_\pi)^{6k^2-6k-n^2+1} \otimes d_\pi(\shL)^{-nk+n(n+1)/2} 
  \otimes (\det \pi_*(\shL\otimes \omega_\pi))^{nk+n(n-1)/2}
 \end{displaymath}

\end{theorem}

\begin{proof}
 Taking into account \ref{notation of picard universal jacobian big genus}, everything follows from \cite[Theorem A and 5.2, Notation 1.5, Remark 5.3]{melo-viviani}.
\end{proof}

\subsection{Genus zero case.}
In this subsection we consider $g=0$, while $d$ is any integer. We will prove the following:

\begin{prop}\label{picard jacobian genus zero}
 Let $\pi\colon \shP\arr \jac_{d,0}$ be the universal curve and $\shL$ be the universal invertible sheaf
 over $\shP$. If $d$ is even then $\Pic \jac_{d,0}$ is freely generated by $\shL_0=\pi_*(\shL\otimes \omega_\pi^{d/2})$ and we have an isomorphism
 \begin{displaymath}
  \det \pi_*(\shL^n \otimes \omega_\pi^k) \simeq \shL_0^{n\max\{nd-2k+1,0\}} \text{ for all } n,k\in \Z
\end{displaymath}
If $d$ is odd then $\Pic \jac_{d,0}$ is freely generated by $\shL_0 = \det \pi_* (\shL \otimes \omega_\pi^{(d-1)/2})$ and we have an isomorphism
\begin{displaymath}
  \det \pi_*(\shL^n \otimes \omega_\pi^k) \simeq \shL_0^{n\max\{nd-2k+1,0\}/2} \text{ for all } n,k\in \Z
\end{displaymath}
\end{prop}
We will need the following lemma, whose proof is standard and thus omitted.

\begin{lemma}\label{characters of GL two and PGL two}
 The group $\duale{(\GL_2)}$ is freely generated by $\det \colon \GL_2\arr \Gm$, while $\duale{(\PGL_2)}=0$
\end{lemma}


\begin{remark}\label{degree zero in genus zero}
 Let $\pi\colon \shP\arr \stS$ be a curve of genus $0$ over an algebraic stack and $\shQ$ be an invertible sheaf on $\shP$. If $\shQ$ has degree $0$, by
 \ref{cbs for elliptic curves} it follows that $\pi_* \shQ$ is an invertible sheaf, it satisfies base change and the map
 $\pi^* \pi_* \shQ\arr \shQ$ is an isomorphism because it is so on the geometric fibers.
 
 If $\shQ$ has degree $1$, by \ref{cbs for elliptic curves} it follows that $\pi_*\shQ$ is a rank $2$ locally free sheaf, it satisfies base change and
 the map $\pi^*\pi_* \shQ \arr \shQ$ is surjective because it is so on the geometric fibers. In particular we obtain an isomorphism
 $\shP \arr \PP(\pi_*\shQ)$: the pullback of $\odi {\PP(\pi_*\shQ)}(1)$ is $\shQ$ and therefore we get the Euler sequence
 \begin{displaymath}
  0\arr \omega_\pi \otimes \shQ \arr \pi^* \pi_* \shQ \arr \shQ \arr 0
 \end{displaymath}
\end{remark}

\begin{proof}[Proof of \ref{picard jacobian genus zero}]
Since $\stM_0 \simeq \Bi \PGL_2$, by \ref{picard of BG} and \ref{characters of GL two and PGL two} we obtain $\Pic \stM_0 = 0$. In particular $\det \pi_*(\omega_\pi^k)$
is trivial for all $k\in \Z$ in $\Pic \jac_{d,0}$.

Assume $d$ even. Tensoring by $\omega_\pi^{d/2}$ yields an isomorphism $\jac_{d,0}\arr \jac_{0,0}$ over $\stM_0$. By \ref{degree zero in genus zero} we see that the functors
 $\jac_{0,0} \arr \Bi_{\stM_0}\Gm$ mapping $(P\arrdi{q} S,\shQ)$ to $(P,q_* \shQ)$ and $\Bi_{\stM_0}\Gm\arr \jac_{0,0}$ mapping $(P\arrdi{q}S,\shT)$ to
 $(P,q^*\shT)$ are quasi-inverses of each other. Moreover by \ref{picard of BG} we have $\Pic \Bi_{\stM_0} \Gm \simeq \Pic \stM_0 \oplus \Z\simeq \Z$,
 generated by the invertible sheaf given by the rule $(P,\shT)\longmapsto \shT$. The pullback of this sheaf via $\jac_{d,0}\simeq \jac_{0,0} \simeq \Bi_{\stM_0}\Gm$ is isomorphic to $\shL_0$, which therefore freely generates $\Pic \jac_{d,0}$. By \ref{degree zero in genus zero} we have $\shL\simeq \pi^* \shL_0 \otimes \omega_\pi^{-d/2}$ and, using projection formula,
\begin{displaymath}
 \det \pi_*(\shL^n \otimes \omega_\pi^k) \simeq \det (\shL_0^n\otimes \pi_*(\omega_\pi^{k-nd/2}))\simeq \shL_0^{n\rk \pi_*(\omega_\pi^{k-nd/2})}
\end{displaymath}
Finally by \ref{cbs for elliptic curves} we see that $\rk \pi_*(\omega_\pi^{k-nd/2})=\max\{nd-2k+1,0\}$.

Assume now that $d$ is odd. Tensoring by $\omega_\pi^{(d-1)/2}$ we get an isomorphism $\jac_{d,0}\arr \jac_{1,0}$. By \ref{degree zero in genus zero} we see that
the functors $\jac_{1,0} \arr \Bi \GL_2$ mapping $(P\arrdi q S,\shQ)$ to $q_*\shQ$ and $\Bi \GL_2 \arr \jac_{1,0}$ mapping $\E$ to $(\PP(\E),\odi {\PP(\E)}(1))$
are quasi-inverses of each other. Moreover by \ref{picard of BG} we have $\Pic \Bi \GL_2 \simeq \Z$ generated by the invertible sheaf given by the rule
$\E\longmapsto \det \E$. The pullback of this sheaf via $\jac_{d,0}\simeq\jac_{1,0}\simeq\Bi \GL_2$ is isomorphic to $\shL_0$. Set $\shT=\shL\otimes \omega_\pi^{(d-1)/2}$, so that $\shL_0 = \det \pi_* \shT$. Applying the determinant to
the Euler sequence in \ref{degree zero in genus zero} associated with $\shT$ we get an isomorphism
\begin{displaymath}
 \omega_\pi \simeq \pi^* \shL_0 \otimes \shT^{-2}
\end{displaymath}
Writing $\shL^n \otimes \omega_\pi ^k$ in terms of $\shT$ and $\shL_0$ and applying projection formula we obtain
\begin{displaymath}
 \det \pi_*(\shL^n \otimes \omega_\pi^k) \simeq \det(\shL_0^{k-n(d-1)/2}\otimes \pi_*(\shT^{nd-2k}))\simeq \shL_0^{(k-n(d-1)/2)\rk \pi_*(\shT^{nd-2k})}
 \otimes \det \pi_*(\shT^{nd-2k})
\end{displaymath}
By \ref{cbs for elliptic curves} we have that $\rk \pi_*(\shT^{nd-2k})=\max\{nd-2k+1,0\}$. Thus it suffices to prove the expression 
$\det \pi_*(\shT^q) \simeq \shL_0^{q(q+1)/2}$ for $q\geq 0$. Considering the Euler sequence in \ref{degree zero in genus zero} associated with $\shT$, 
replacing $\omega_\pi$ by $\pi^*\shL_0\otimes \shT^{-2}$ and tensoring by $\shT^q$ we get an exact sequence
\begin{displaymath}
 0\arr \pi^* \shL_0 \otimes \shT^{q-1} \arr \pi^* \pi_* \shT \otimes \shT^q \arr \shT^{q+1}\arr 0
\end{displaymath}
The pushforward $\pi_*$ of the above sequence for $q\geq 0$ is exact because $\R^1\pi_*(\pi^* \shL_0 \otimes \shT^{q-1})=0$ thanks to
\ref{cbs for elliptic curves}. Thus applying $\pi_*$, the determinant, the projection formula and using that $\rk \pi_* (\shT^r)=r+1$ for $r\geq -1$
we get an isomorphism
\begin{displaymath}
 \det \pi_*(\shT^{q+1})\simeq \shL_0 \otimes (\det \pi_*(\shT^q))^2\otimes (\det \pi_*(\shT^{q-1}))^{-1}
\end{displaymath}
It is now easy to check by induction that $\det \pi_*(\shT^q)\simeq \shL_0^{q(q+1)/2}$.
\end{proof}

\subsection{Genus one case.}
In this subsection we consider $g=1$ and $d>0$.
We will prove the following Theorem.

\begin{theorem}\label{picard group of Xd}
 Let $\pi\colon \E\arr\jac_{d,1}$ be the universal curve over $\jac_{d,1}$, $\shL$ be the universal invertible sheaf over $\E$ and assume $p\nmid d$. Then 
 $\Pic \jac_{d,1}$ is generated by $\pi_* \omega_\pi$ and $\det \pi_*\shL$ with the only relation $(\pi_*\omega_\pi)^{12}$.
 Moreover we have an isomorphism
 \[
 \det \pi_*(\shL^n \otimes \omega_\pi^k) \simeq (\det \pi_*\shL)^{n^2}\otimes (\pi_*\omega_\pi)^{dnk+(n-1)(dn-2n-2)/2}\text{ for }n>0,k\in \Z
 \]
\end{theorem}

The starting point is the well known Theorem of Mumford,
later generalized by Fulton and Olsson. (See \cite{Mumford1963} and \cite{Fulton2010})
\begin{theorem}[Mumford, Fulton, Olsson]\label{Mumford's theorem}
 The Picard group of $\stM_{1,1}$ is cyclic of order $12$ and it is generated by $\pi_*\omega_\pi$, where
 $\pi\colon \E\arr\stM_{1,1}$ is the universal curve.
\end{theorem}

We will proceed by showing that $\jac_{d,1}$ is isomorphic to $\Bi_{\stM_{1,1}} G_d$, for a certain group scheme $G_d$ over $\stM_{1,1}$.
In particular we will conclude that $\Pic\jac_{d,1} \simeq \Pic \stM_{1,1}\oplus \Hom(G_d,\Gm)$ and we will conclude the section by computing the group
of characters $\Hom(G_d,\Gm)$.

\begin{lemma}\label{elliptic curves and pic}
  Let $\pi\colon \E \arr \stS$ be a genus one curve over an algebraic stack. Then the functor
  \[
  \begin{tikzpicture}[xscale=4.0,yscale=-0.6]
    \node (A0_0) at (0, 0) {$\E$};
    \node (A0_1) at (1, 0) {$\Picsh^1_{\E/\stS}$};
    \node (A1_0) at (0, 1) {$(\delta\colon T\arr \E\times_\stS T)$};
    \node (A1_1) at (1, 1) {$[\odi{\E\times_\stS T}(\delta)]$};
    \path (A0_0) edge [->]node [auto] {$\scriptstyle{\Omega}$} (A0_1);
    \path (A1_0) edge [|->,gray]node [auto] {$\scriptstyle{}$} (A1_1);
  \end{tikzpicture}
  \]
is an isomorphism. If $\sigma\colon \stS\arr \E$ is a section, then also the functor
  \[
  \begin{tikzpicture}[xscale=4.0,yscale=-0.6]
    \node (A0_0) at (0, 0) {$\E$};
    \node (A0_1) at (1, 0) {$\Picsh^0_{\E/\stS}$};
    \node (A1_0) at (0, 1) {$(\delta\colon T\arr E\times_\stS T)$};
    \node (A1_1) at (1, 1) {$[\odi{\E\times_\stS T}(\delta-\sigma\times_\stS T)]$};
    \path (A0_0) edge [->]node [auto] {$\scriptstyle{}$} (A0_1);
    \path (A1_0) edge [|->,gray]node [auto] {$\scriptstyle{}$} (A1_1);
  \end{tikzpicture}
  \]
is an isomorphism that sends $\sigma$ to $[\odi \E]$
\end{lemma}
\begin{proof}
 The last part of the statement follows from the first one. We start showing that $\Omega$ is an fppf epimorphism.
 Let $\chi\in \Picsh^1_{\E}(T)$, where $T$
is an $\stS$-scheme. We can replace $T$ by $\stS$ and assume the existence of an invertible sheaf $\shL$ of degree $1$ over $\E$
such that $\chi=[\shL]$. By \ref{cbs for elliptic curves} the sheaf
$\shQ=\pi_*\shL$ is invertible and we can assume it is trivial. Since $\pi_*\shL$ satisfies base change, the section
$s$ of $\shL$ corresponding to $1\in\pi_*\shL$ is always non-zero on the geometric fibers of $\pi$. Thus, by \ref{degree one sheaves and sections},
we get a section $\tau\colon \stS\arr\E$ with an isomorphism $\odi\E(\tau)\simeq \shL$, so that $\Omega(\tau)=\chi$.

We now prove that $\Omega$ is injective over an $\stS$-scheme $T$. Again we can replace $T$ by $\stS$, so that, in particular,
$\stS$ is a scheme. Let $\delta,\delta'\in \E(\stS)$ such that $[\odi \E(\delta')]=[\odi \E(\delta)]$.
By \ref{cbs for elliptic curves} we have an isomorphism 
\[
\odi \E(\delta)\arrdi \alpha \odi \E(\delta')\otimes \pi^*\shQ
\]
for some invertible sheaf $\shQ$ over $\stS$. Since $\pi_*\odi\E(\delta)$ and $\pi_*\odi\E(\delta')$ are freely generated 
by the respective sections $1$ by \ref{degree one sheaves and sections}, applying $\pi_*$ we get an element $x\in \shQ$ such that $\alpha(1)=1\otimes \pi^*x$ and that
freely generates $\shQ$. Thus we obtain an isomorphism $\odi \E(\delta)\arr\odi \E(\delta')$ mapping $1$ to $1$ and therefore that
$\delta=\delta'$.
\end{proof}

\begin{prop}\label{computation of line bundles on the universal elliptic curve}
 Let $\E\arrdi \pi \stS$ be a genus one curve over an algebraic stack with a section $\sigma$. Then we have isomorphisms
 \[
 \det \pi_*\odi \E(n\sigma)\simeq \sigma^* \odi \E(\sigma)^{n(n+1)/2-1} \text{ for } n\geq 1
 \]
 and
 \[
 \sigma^*\odi \E(\sigma)\simeq \duale{(\pi_*\omega_\pi)}
 \]
 Moreover if $\shL$ is an invertible sheaf over $\E$ of degree greater than zero and $\tau$ is another section (possibly equal to $\sigma$) we have an isomorphism
 \[
 \det \pi_*(\shL\otimes \odi \E(\tau)) \simeq \det \pi_*(\shL\otimes\odi \E(\tau-\sigma))\otimes \sigma^*\shL\otimes \sigma^*\odi \E(\tau)
 \]
\end{prop}
\begin{proof}
In what follows we will use \ref{cbs for elliptic curves} and, in particular, 
the isomorphism $\omega_\pi\simeq\pi^*\pi_*\omega_\pi$ without further comments.
Notice that the first isomorphism in the statement follows from the last one with $\tau=\sigma$, induction and the isomorphism
$\pi_*\odi \E(\sigma) \simeq \odi \stS$ (see \ref{degree one sheaves and sections}).

Consider an invertible sheaf $\shL$ on $\E$ of degree greater or equal than zero and a section $\tau$ of $\E$.
Tensoring by $\shL\otimes\odi \E(\tau)$ the exact sequence on $\E$
\[
0\arr \odi \E(-\sigma)\arr \odi \E \arr \sigma_*\odi \stS \arr 0 
\]
and applying $\pi_*$ we get an exact sequence
\[
0\arr \pi_*(\shL\otimes \odi \E(\tau-\sigma))\arr \pi_*(\shL\otimes \odi \E(\tau))\arr \sigma^*(\shL\otimes \odi \E(\tau))\arr 
\R^1\pi_*(\shL\otimes\odi \E(\tau-\sigma)) \arr  0
\]
If $\shL=\odi \E$ and $\tau=\sigma$, we obtain a surjection $\sigma^*\odi \E(\sigma)\arr \R^1\pi_*\odi \E\simeq \duale{(\pi_*\omega_\pi)}$ 
which is therefore an isomorphism. If $\shL$ has degree
strictly greater than $0$, the last term in the sequence is zero and taking the determinant we get the last isomorphism.
\end{proof}

\begin{prop}\label{fundamental map}
 The functor
  \[
  \begin{tikzpicture}[xscale=3.0,yscale=-0.6]
    \node (A0_0) at (0, 0) {$\stM_{1,1}$};
    \node (A0_1) at (1, 0) {$\jac_{d,1}$};
    \node (A1_0) at (0, 1) {$(E,\sigma)$};
    \node (A1_1) at (1, 1) {$(E,\odi E(d\sigma))$};
    \path (A0_0) edge [->]node [auto] {$\scriptstyle{}$} (A0_1);
    \path (A1_0) edge [|->,gray]node [auto] {$\scriptstyle{}$} (A1_1);
  \end{tikzpicture}
  \]
is an epimorphism in the fppf topology and it is a section of the functor
  \[
  \begin{tikzpicture}[xscale=3.0,yscale=-0.6]
    \node (A0_0) at (0, 0) {$\jac_{d,1}$};
    \node (A0_1) at (1, 0) {$\stM_{1,1}$};
    \node (A1_0) at (0, 1) {$(E,\shL)$};
    \node (A1_1) at (1, 1) {$(\Picsh^0_{E/S},[\odi E])$};
    \path (A0_0) edge [->]node [auto] {$\scriptstyle{}$} (A0_1);
    \path (A1_0) edge [|->,gray]node [auto] {$\scriptstyle{}$} (A1_1);
  \end{tikzpicture}
  \]
\end{prop}
\begin{proof}
 The second part of the statement follows from \ref{elliptic curves and pic}. For the first one let $(E,\shL)\in \jac_{d,1}$. 
 By \ref{multiplication by two on jacobians} fppf locally we can write $\shL\simeq \shT^d$ for some degree $1$ invertible
 sheaf $\shT$ on $E$. By \ref{elliptic curves and pic} $E$ has a section $\sigma$ such that $[\shT]=[\odi E(\sigma)]$, which means that
 $\shT$ (resp. $\shL$) and $\odi E(\sigma)$ (resp. $\odi E(d\sigma)$) are fppf locally isomorphic.
\end{proof}

\begin{lemma}\label{torsors over a stack}
 Let $p\colon \stX\arr \stY$ be a map of stacks over a scheme $S$ with a section $s\colon \stY\arr\stX$ which is an fppf epimorphism and denote by
 $G$ the sheaf of groups over $\stY$ defined by
$$
  G(T\arrdi \xi \stY)=\Ker(\Aut_\stX(s(\xi))\arr \Aut_\stY(p s(\xi)))
$$
Then the functor $F\colon \stX\arr \Bi_\stY G$ which maps $\eta\colon T\arr\stX$ to the sheaf $F(\eta)$ over $T$ given by the inverse image of the identity section of $\Autsh_T(p(\eta))$ along the map
\begin{displaymath}
 \Isosh_T (\eta,sp(\eta))\arrdi p \Isosh_T(p(\eta),ps(p(\eta)))\simeq \Autsh_T (p(\eta))
\end{displaymath}
defines an equivalence of stacks.
\end{lemma}
\begin{proof}
Given $\xi\colon T\arr\stY$ denote by $\stX_\xi$, $G_\xi$ and $s_\xi\colon T\arr \stX_\xi$ the base change of $\stX$, $G$ and $s\colon \stY\arr\stX$
along $\xi$. Given $\eta\colon T\arr \stX$, we show that $F(\eta)$ is a $G$-torsor. A direct check shows that $F(\eta)$ coincides with $\Isosh_{\stX_{p(\eta)}}(\eta,s_{p(\eta)})$ and
$G_\xi\simeq \Autsh_{\stX_\xi}(s_\xi)$.
Since $\eta$ and $s_{p(\eta)}$ are fppf locally isomorphic, it follows that $F(\eta)$ is a $G_{p(\eta)}$-torsor. It is easy to see that the association
$\eta\longmapsto F(\eta)$ defines a a functor
$F\colon \stX\arr \Bi_\stY G$.
Since the base change of $F$ along any morphism $T\arr\stY$ is an equivalence by standard results,
we obtain that it is globally an equivalence.
\end{proof}

\begin{definit}
We define the group functor $G_d$ over $\stM_{1,1}$ as the group $G$ obtained as in
\ref{torsors over a stack} with respects to the maps defined in \ref{fundamental map}.
\end{definit}

Let us describe the group $G_{d}$ more concretely.

\begin{prop}\label{properties of Gd}
An element of $G_d(S\arrdi{(E,\sigma)}\stM_{1,1})$ is a pair $(f,\lambda)$ where $f\colon E\arr E$ is a translation by an element in $E[d]$ and 
$\lambda\colon \odi E(\sigma)^d\arr \odi E(f(\sigma))^d$ is an isomorphism. Moreover we have an exact sequence
  \[
  \begin{tikzpicture}[xscale=1.5,yscale=-0.6]
    \node (A0_2) at (2, 0) {$(f,\lambda)$};
    \node (A0_3) at (3, 0) {$f(\sigma)$};
    \node (A1_0) at (0, 1) {$0$};
    \node (A1_1) at (1, 1) {$\Gm$};
    \node (A1_2) at (2, 1) {$G_d$};
    \node (A1_3) at (3, 1) {$\E[d]$};
    \node (A1_4) at (4, 1) {$0$};
    \node (A2_1) at (1, 2) {$\mu$};
    \node (A2_2) at (2, 2) {$(\id,\mu)$};
    \path (A2_1) edge [|->,gray]node [auto] {$\scriptstyle{}$} (A2_2);
    \path (A1_0) edge [->]node [auto] {$\scriptstyle{}$} (A1_1);
    \path (A1_1) edge [->]node [auto] {$\scriptstyle{}$} (A1_2);
    \path (A1_2) edge [->]node [auto] {$\scriptstyle{}$} (A1_3);
    \path (A0_2) edge [|->,gray]node [auto] {$\scriptstyle{}$} (A0_3);
    \path (A1_3) edge [->]node [auto] {$\scriptstyle{}$} (A1_4);
  \end{tikzpicture}
  \]
in the Zariski topology of $\Sch/\stM_{1,1}$, where $\E\arr \stM_{1,1}$ is the universal curve. 
In particular $G_d$ is affine and of finite type over $\stM_{1,1}$ and, if $p\nmid d$, it is smooth over $\stM_{1,1}$.
\end{prop}
\begin{proof}
 By definition, an element of $G_d(S\arrdi{(E,\sigma)}\stM_{1,1})$ is a pair $(f,\lambda)$ where $f\colon E\arr E$ is an isomorphism
such that $f_*\colon \Picsh^0_{E/S}\arr \Picsh^0_{E/S}$ is the identity and 
$\lambda\colon \odi E(\sigma)^d\arr \odi E(f(\sigma))^d$ is an isomorphism. Taking into account \ref{elliptic curves and pic}, the condition 
$f_*=\id$ means that, for all $\delta\in \E(T)$ and $S$-scheme $T$ we have
$$
  [\odi{E_T}(f(\delta)-f(\sigma))]=[\odi{E_T}(\delta-\sigma)]
$$
which implies that $f\colon E\arr E$ is a translation. The existence of $\lambda$ also implies that $f(\sigma)\in E[d]$.
In particular the sequence in the statement is well defined and, since $\pi_*\mathbb{G}_{m,E}\simeq \mathbb{G}_{m,S}$, it is exact in the first two terms.

It remains to prove that $G_d\arr \E[d]$ is a Zariski epimorphism. This will also imply that $G_d$ is locally a product
of $\Gm$ and $\E[d]$, proving the last sentence in the statement. Let $\delta \in \E[d]$ and $t$ be the translation by $\delta$, so that $t(\sigma)=\delta$.
By \ref{cbs for elliptic curves}, the equality $[\odi E(\delta-\sigma)^d]=0$
in $\Picsh^0_{E/S}$ implies that $\odi E(\delta-\sigma)^d\simeq \pi^*\shQ$, where
$\shQ$ is an invertible sheaf over $S$. In particular, where $\shQ$ is trivial, 
we get an isomorphism $\lambda\colon \odi E(\sigma-t(\sigma))^d \arr \odi E$ and therefore a
pair $(t,\lambda)\in G_d$ over $\delta$.
\end{proof}
The groups obtained from $G_d$ as base change along geometric points $\Spec k\arrdi{(E,p)}\stM_{1,1}$ are particular cases of Theta groups,
first defined by Mumford in his paper \cite{Mumford1966}. With notation from this paper we have $G_d\times_{\stM_{1,1}} k = \shG(\odi E(dp))$.

\begin{cor}\label{structure of Xd}
We have an isomorphism $\jac_{d,1} \simeq \Bi_{\stM_{1,1}} G_d$.
\end{cor}

By \ref{picard of BG} the last step in the computation of the Picard group of $\jac_{d,1}$ is the study of the group of characters $\Hom(G_d,\Gm)$.
For the remaining part of this subsection we assume $p\nmid d$.

\begin{lemma}\label{torsion groups are autodual}
 The map
  \[
  \begin{tikzpicture}[xscale=3.0,yscale=-0.6]
    \node (A0_0) at (0, 0) {$G_d\times G_d$};
    \node (A0_1) at (1, 0) {$\Gm$};
    \node (A1_0) at (0, 1) {$(x,y)$};
    \node (A1_1) at (1, 1) {$xyx^{-1}y^{-1}$};
    \path (A0_0) edge [->]node [auto] {$\scriptstyle{\hat e_d}$} (A0_1);
    \path (A1_0) edge [|->,gray]node [auto] {$\scriptstyle{}$} (A1_1);
  \end{tikzpicture}
  \]
is bilinear and factors through a map $e_d\colon \E[d]\times \E[d]\arr \Gm$, where $\E\arr \stM_{1,1}$ is the universal curve.
Moreover the map $e_d$ induces an isomorphism
$$
  \E[d]\simeq \Homsh(\E[d],\Gm)
$$
\end{lemma}
\begin{proof}
 The map $\hat e_d$ is well defined because $\E[d]$ is abelian. Note that $\Gm\subseteq G_d$ is contained in the center. In particular the map $e_d$ in the 
statement is well defined and we have
\begin{align*}
 \hat e_d(xy,z)=xyzy^{-1}x^{-1}z^{-1}=x\hat e_d(y,z)zx^{-1}z^{-1}=\hat e_d(x,z)\hat e_d(y,z)
\end{align*}
Finally $\hat e_d(x,y)=\hat e_d(y,x)^{-1}$ and $\hat e_d$ and $e_d$  are therefore bilinear. Let $\phi$ be the induced map $\E[d]\arr \Homsh(\E[d],\Gm)$. This is
a map between flat and finite group schemes and we can check that it is an isomorphism on the geometric points. So let $E$ be an elliptic curve over an algebraically closed
field. In this case $E[d]\simeq \Z/d\Z\times\Z/d\Z$ and $\Homsh(E[d],\Gm)\simeq E[d]$ as abstract groups. 
The result then follows from the fact that $e_d$ is non-degenerate thanks to \cite[§1, Theorem 1]{Mumford1966}.
\end{proof}

\begin{prop}\label{torsion elliptic has no sections}
Let $\E\arr \stM_{1,1}$ be the universal curve. Then 
$$
  \duale{\E[d]} = \Hom(\E[d],\Gm)=0
$$
In particular the map $\duale G_d\arr \duale{\mathbb{G}}_m$ induced by the inclusion $\Gm\arr G_d$ is injective.
\end{prop}
\begin{proof}
By \ref{torsion groups are autodual} we have $\Homsh(\E[d],\Gm)\simeq \E[d]$. 
Therefore we have to prove that there are no sections $\stM_{1,1}\arr \E[d]$ but the zero one.
By contradiction, assume we have a non zero section $\delta\colon \stM_{1,1}\arr \E[d]$. In particular, by base change, we have
a section $\delta_E$ for all elliptic curves $E$. Since $\E[d]$ is \'etale and separated over $\stM_{1,1}$, there are no
elliptic curves $E$ such that $\delta_E$ is zero. Moreover we can assume that $d$ is prime.
If $d=2$, let $k$ be a field having an irreducible and separable polynomial $g\in k[x]$ of degree $3$ and consider the elliptic curve
defined by the equation $y^2=g(x)$. In this case $E[2](k)=0$ by \cite[Group law algorithm 2.3]{Silverman1986} because $g$ has no zeros in $k$
and therefore $\delta_E =0$, which is a contradiction.

So we can assume that $d$ is an odd prime. Let $k$ be a field having a degree $2$ separable extension $L$ and an elliptic curve $E$. We are going to prove that 
$\delta_E\in E(k)$ is invariant under the involution of $E$, that is that $\delta_E \in E[2](k)$.
This will end the proof because $E[d](k)\cap E[2](k)=0$.

We want to construct a new elliptic curve $E'$ over $k$ with the following construction. Let $F$ be a sheaf of sets over a scheme $S$ with an involution
$i\colon F\arr F$, let $P\arr S$ be a $\Z/2\Z$-torsor and call $\sigma\colon P\arr P$ the induced order $2$ automorphism.
We define a new functor by
\begin{displaymath}
 F_{P/S,i}\colon (\Sch/S)^{\text{op}} \arr \sets\comma  F_{P/S,i}(T)=\{x\in F(T\times P) \st \sigma^*x=i(x)\}
\end{displaymath}
Since $i$ and $\sigma^*$ commute, it is easy to check that $F_{P/S,i}$ is a sheaf and that if $S'\arr S$ is any base change then
$F_{P/S,i}\times S' = F_{P\times S'/S',i}$. Moreover, if $\Delta\colon T\arr P\times T$ is a section it is easy to check that
\[
F_{P/S,i}(T)\arr F(T)\comma z\longmapsto \Delta^*z
\]
is a bijection.
In particular for any torsor $P$ we get an isomorphism $\tau\colon F_{P/S,i}\times P\arr F\times P$ induced by the diagonal section $P\arr P\times P$.
We now claim that if there exist $x\in F_{P/S,i}(S)$ and $y\in F(S)$ such that $\tau(h^*x)=h^*y$, where $h\colon P\arr S$ is the structure
morphism, then it follows that $i(y)=y$. Indeed by construction $\tau(h^*x)=x$, $\sigma^*h^*y=h^*y$ and, since $x\in F_{P/S,i}(S)$, $\sigma^*x=i(x)$.
In particular $h^*i(y)=h^*y$ and by descent $i(y)=y$.

We apply the previous construction with $F=E$, $i$ the involution of $E$ and $P=\Spec L$. Set $E'=F_{P/k,i}$. The sheaf $E'$ is a genus one curve and
$0\in E'(k)\subseteq E(L)$. Thus $E'$ is an elliptic curve and, by construction, $\tau\colon E'\times P\arr E\times P$ preserves the neutral element,
that is it is an isomorphism of elliptic curves. Since $\delta_E$ and $\delta_{E'}$ come from a global section on $\stM_{1,1}$ we should have
$\tau(\delta_{E'\times P})=\delta_{E\times P}$, which implies that $i(\delta_E)=\delta_E$ as shown above.
\end{proof}

\begin{prop}\label{computation of det pi star L to the n}
Consider the map $\duale G_d \arr \duale{ \mathbb{G}}_m = \Z$ induced by the inclusion $\Gm\arr G_d$. The morphism
$$
  \Pic \jac_{d,1}\simeq \Pic \stM_{1,1} \oplus \duale G_d\arr \Pic \stM_{1,1} \oplus \Z \simeq \Z/12\Z\oplus \Z
$$
 sends the invertible sheaf on $\jac_{d,1}$
defined by $\jac_{d,1}\ni (E\arrdi \pi S,\shL)\longmapsto \det \pi_* \shL^n$ to the element $(1-nd(nd+1)/2,n^2d)$
\end{prop}
\begin{proof}
 Notice that $\Bi_{\stM_{1,1}} \Gm$ is isomorphic to the stack $\stX$ of triples $(E\arrdi \pi S,\sigma,\shQ)$
 where $(E,\sigma)\in \stM_{1,1}$ and $\shQ$ is an invertible sheaf over $S$.
 Consider the functor
  \[
  \begin{tikzpicture}[xscale=5.0,yscale=-0.6]
    \node (A0_0) at (0, 0) {$\stX$};
    \node (A0_1) at (1, 0) {$\jac_{d,1}$};
    \node (A1_0) at (0, 1) {$(E\arrdi \pi S,\sigma,\shQ)$};
    \node (A1_1) at (1, 1) {$(E \arrdi \pi S,\odi E(d\sigma)\otimes \pi^* \shQ)$};
    \path (A0_0) edge [->]node [auto] {$\scriptstyle{\Omega}$} (A0_1);
    \path (A1_0) edge [|->,gray]node [auto] {$\scriptstyle{}$} (A1_1);
  \end{tikzpicture}
  \]
  We claim that the functor $\widetilde \Omega\colon \Bi_{\stM_{1,1}}\Gm\simeq \stX\arr \jac_{d,1}\simeq \Bi_{\stM_{1,1}}G_d$ is
  naturally equivalent to the map $\beta_*\colon \Bi_{\stM_{1,1}}\Gm\arr \Bi_{\stM_{1,1}}G_d$ induced by the inclusion
  $\Gm\arr G_d$.
  Given $(E\arrdi \pi S,\sigma,\shQ)\in \stX$, the associated $\Gm$-torsor is 
  $\Isosh_S(\shQ,\odi S)$, while $\widetilde \Omega(\Isosh_S(\shQ,\odi S))$ is the subsheaf of
  $$\Isosh_{\jac_{d,1}}((E,\odi E(d\sigma)\otimes\pi^*\shQ),(E,\odi E(d\sigma)))$$ 
  of isomorphism $(f,\lambda)$ such that 
  $f_*\colon \Pic^0_E\arr \Pic^0_E$ is the identity by \ref{torsors over a stack}.
  In particular there is a $\Gm$-equivariant map $\Isosh_S(\shQ,\odi S)\arr \widetilde \Omega(\Isosh_S(\shQ,\odi S))$
  compatible with base changes and automorphisms of $\shQ$.
  Recall that $\beta_*(P)=(P\times G_d)/\Gm$ for a $\Gm$-torsor $P$ and that, if $Q$ is a $G_d$-torsor,
  $\Gm$-equivariant maps $P\arr Q$ are in one to one corresponce with $G_d$-equivariant isomorphism
  $\beta_*(P)\arr Q$. Thus given $P$, the $\Gm$-equivariant map $P\arr \widetilde \Omega(P)$ induces
  an isomorphism $\beta_*(P)\arr \widetilde \Omega(P)$ of $G_d$-torsors and it is easy to check that
  this yields a natural isomorphism $\beta_*\arr \widetilde \Omega$.
  
  We can conclude that the map in the statement can be seen as the pull-back along $\Omega\colon \stX\arr\jac_{d,1}$ on the Picard groups. The invertible sheaf
defined in the statement is sent to
$$
  (E\arrdi \pi S,\sigma, \shQ)\longmapsto \det \pi_*(\odi E(nd\sigma) \otimes \pi^* \shQ^n) \simeq (\det \pi_* \odi E(nd\sigma)) \otimes \shQ^{n^2d}
$$
Since the invertible sheaf $(E,\sigma,\shQ)\longmapsto \shQ$ is the generator of $\duale{\mathbb G}_m$ in $\Pic \Bi_{\stM_{1,1}} \Gm$, the result follows
from \ref{computation of line bundles on the universal elliptic curve}.
\end{proof}

\begin{prop}\label{image of dual Gd is dZ}
 The image of the map $\duale G_d\arr \duale{\mathbb{G}}_m \simeq \Z$ induced by the inclusion $\Gm\arr G_d$ is $d\Z$.
\end{prop}
\begin{proof}
 Call $\alpha\colon  \duale G_d\arr \duale{\mathbb{G}}_m \simeq \Z$ the map in the statement.
By \ref{computation of det pi star L to the n} applied to $n=1$, we see that $d\Z\subseteq \Imm \alpha$.
Let $E$ be an elliptic curve over an algebraically closed field $k$. By \ref{torsion groups are autodual}, it follows that there
exist $x,y\in G_d(k)$ such that $\omega=e_d(x,y)=xyx^{-1}y^{-1}$ is a primitive $d$-root of unity. If $\psi\colon G_d\arr \Gm$ is a morphism and we set $r=\alpha(\psi)$,
 we have $1=\psi(\omega)=\omega^r$, which implies that $d\mid r$.
\end{proof}

\begin{proof}[Proof of theorem \ref{picard group of Xd}]
 By \ref{structure of Xd} we have $\jac_{d,1}\simeq\Bi_{\stM_{1,1}}G_d$ and by \ref{picard of BG} we can conclude that $\Pic\jac_{d,1}\simeq\Pic\stM_{1,1}\oplus \duale{G}_d$.
 By \ref{torsion elliptic has no sections} and \ref{image of dual Gd is dZ} we have $\duale G_d=d\duale{\mathbb{G}}_m \subseteq \duale{\mathbb{G}}_m$
 and by \ref{computation of det pi star L to the n} that
 $$
 \shT=\det\pi_*\shL \otimes (\pi_*\omega_\pi)^{d(d+1)/2 - 1}
 $$
 freely generates $\duale G_d$. Again by \ref{computation of det pi star L to the n} we have
 \[
 \det\pi_*(\shL^n)\simeq \shT^{n^2}\otimes (\pi_*\omega_\pi)^{1-nd(nd+1)/2}
 \]
 Taking into account that $\omega_\pi \simeq \pi^* \pi_* \omega_\pi$ and using projection formula, a direct computation concludes the proof.
\end{proof}

\section{Canonical covers and their discriminant loci.}\label{Canonical covers and their discriminant loci}

In this section we work over a field of characteristic $p\geq 0$ and we fix non negative integers $g$ and $d$ and a positive integer $n$.

\begin{definit}\label{description of Udn and the cover Hdn}
We denote by $\stV_{d,g,n}$ (resp. $\stU_{d,g,n}$) the stack of triples $(p\colon C\arr S,\shQ,s)$ where $(p\colon C\arr S,\shQ)\in \jac_{d,g}$ and $s$ is a section of $\shQ^n$
(resp. that is not identically zero on any of the geometric fibers of $p\colon C\arr S$).

Let $\shC\arr \stU_{d,g,n}$ be the universal curve of $\stU_{d,g,n}$ and $\shL$ be its universal invertible sheaf. By definition, $\shL^n$ carries a
 section and we will call it the \emph{marked section} of $\shL^n$. 
 
 If $d>0$ we denote by $\stH_{d,g,n}$ the zero locus of the marked section of $\shL^n$.
 By \ref{degree one sheaves and sections}
 the map $\stH_{d,g,n}\arr \stU_{d,g,n}$ is a degree $nd$ cover
 and we will call it the \emph{canonical cover} of $\stU_{d,g,n}$.
 Moreover the closed immersion $\stH_{d,g,n}\arr \shC$ defines a functor $\rho_{d,g,n}\colon \stU_{d,g,n}\arr \Hilb^{nd}_{\stM_{g,1}/\stM_g}$: an object
 $(C,\shQ,s)$ is sent to the zero locus $Z(s) \subseteq C$ of $s\in \shQ^n$. Finally we denote by $\stZ_{d,g,n}$ the discriminant locus of the canonical cover $\stH_{d,g,n}\arr\stU_{d,g,n}$ of $\stU_{d,g,n}$.
\end{definit}

The aim of this section is to understand the geometry of the stacks $\stU_{d,g,n}$ and $\stZ_{d,g,n}$. We will prove the following.

\begin{theorem}\label{key proposition for the bad locus}
 Assume $d>0$. Let $\shC\arrdi \pi \stU_{d,g,n}$ be the universal curve of $\stU_{d,g,n}$ and 
 $\shL$ be the universal invertible sheaf over $\shC$. Then $\stZ_{d,g,n}\neq \emptyset$ if and only if $dn\neq 1$ and in this case $\stZ_{d,g,n}$ is flat and surjective over $\stM_g$ and we have:
 \begin{enumerate}
 \item \label{key proposition for the bad locus: zero} $\stZ_{d,g,n}$ is the zero locus in $\stU_{d,g,n}$ of a section of the invertible sheaf
 \[
  (\det \pi_*(\shL^n\otimes \omega_\pi))^2 \otimes (\det \pi_*\omega_\pi)^{-2}
 \]
 \item \label{key proposition for the bad locus: first} if $p\nmid n$ or $g=0$ then $\stZ_{d,g,n}$ is geometrically reduced over $\stM_g$ and, in particular, reduced;
 \item \label{key proposition for the bad locus: second} $\stZ_{d,g,n}$ is irreducible in the following cases: $n=1$; $dn>2g$; $dn=2g$ and $g\geq 3$; $dn=2g-1$ and $g\geq 4$;
 \item \label{key proposition for the bad locus: third} If $g=1$ and $p\neq 2$ then $\stZ_{1,1,2}$
 is a disjoint union of two integral substacks
 of $\stU_{1,1,2}$, one of which is the zero locus of a section of the invertible sheaf
 \[
 \det(\pi_* \shL^2)\otimes (\pi_*\shL)^{-2}
 \]
\end{enumerate}
\end{theorem}

We start giving a more precise description of the stacks $\stU_{d,g,n}$, proving in particular that they are algebraic and explaining the relation with the Hilbert
scheme $\Hilb^{nd}_{\stM_{g,1}/\stM_g}$.

\begin{remark}\label{when Vdn is a vector bundle}
Assume $nd>2g-2$. In this case, by \ref{cbs for elliptic curves}, the stack
$\stV_{d,g,n}$ is a vector bundle of rank $nd+1-g$ over $\jac_{d,g}$ corresponding to the locally free sheaf $\pi_*(\shL^n)$, where 
$\pi \colon\shC\arr\jac_{d,g}$ is the universal curve and $\shL$ is the universal invertible sheaf over $\shC$. Moreover $\stU_{d,g,n}$ is the complement
of the zero section of $\stV_{d,g,n}$.
\end{remark}

\begin{prop}\label{Ud1 has Hilbert scheme}
 If $d>0$ the functor $\rho_{d,g,1}\colon \stU_{d,g,1}\arr \Hilb^{d}_{\stM_{g,1}/\stM_g}$ is an equivalence. A quasi-inverse is given by
   \[
  \begin{tikzpicture}[xscale=2.7,yscale=-0.7]
    \node (A0_0) at (0, 0) {$\Hilb^{d}_{\stM_{g,1}/\stM_g}$};
    \node (A0_1) at (1, 0) {$\stU_{d,g,1}$};
    \node (A1_0) at (0, 1) {$(C,Z\subseteq C)$};
    \node (A1_1) at (1, 1) {$(C,\duale \shI_Z,1)$};
    \path (A0_0) edge [->]node [auto] {$\scriptstyle{}$} (A0_1);
    \path (A1_0) edge [|->,gray]node [auto] {$\scriptstyle{}$} (A1_1);
  \end{tikzpicture}
  \]
where $\shI_Z$ is the sheaf of ideals defining $Z$.
\end{prop}
\begin{proof}
 Given an object
 $Z\subseteq C \in \Hilb^d_{\stM_{g,1}/\stM_g}(S)$, we have to prove that $\shI_Z$ is invertible over $C$.
 Since $\Hilb^d_{\stM_{g,1}/\stM_g}$ is smooth over $\stM_g$, we can assume that $S$ and therefore $C$ are smooth
 over the base field.
 By flatness of $Z$, $I_Z$ is invertible on the fibers of $C\arr S$ and
therefore $\dim_{k(p)} \shI_Z\otimes k(p)=1$ for all $p\in C$, which implies that $I_Z$ is invertible over $C$.
The above discussion shows that the functor in the statement is well defined. The fact that the functors are quasi-inverses of each other is standard.
\end{proof}

\begin{prop}\label{the zero case and mu n}
 If $(C,\shQ,s)\in \stU_{0,g,n}$ then $s\colon \odi C\arr \shQ^n$ is an isomorphism. Moreover 
 the functor $\stU_{0,g,n}\arr (\Picsh^0_{\stM_{g,1}/\stM_g})[n]$ is a $\mu_n$-gerbe.
\end{prop}
\begin{proof}
 The first claim follows from the fact that a non zero section of a degree $0$ invertible sheaf on a proper curve over an algebraically closed field
 always generates it. Denote by $F$ the functor in the statement. Given $\xi\colon T\arr (\Picsh^0_{\stM_{g,1}/\stM_g})[n]$ 
 denote by $\stY\arr T$ the base change of $F$ along $\xi$. In order to prove that $\stY$ is a $\mu_n$-gerbe, we can assume that $\xi$ is
 given by a curve $p\colon C\arr T$ and the class of some invertible sheaf $\shL$ on $C$. Since $n[\shL]=0$ in $\Picsh^0_C$ the sheaf 
 $p_*(\shL^n)$ is invertible by \ref{cbs for elliptic curves} and we can assume that it is trivial. An object of $\stY(T)$
 is a pair $(\shQ,s)$ where $\shQ$ is an invertible
 sheaf over $C$, $s\colon \odi {C}\arr \shQ^n$ is an isomorphism and $[\shQ]=[\shL]$. Those data define an invertible sheaf 
 $\shT=p_*(\shL\otimes \shQ^{-1})$ with an isomorphism $\odi T \arr \shT^n$. Since $T$ is arbitrary we get a functor $\stY\arr \Bi_T \mu_n$
 which is easily seen to be an equivalence.
\end{proof}

\begin{prop}\label{multiplication by n in Xd}
 The map
   \[
  \begin{tikzpicture}[xscale=2.1,yscale=-0.6]
    \node (A0_0) at (0, 0) {$\jac_{d,g}$};
    \node (A0_1) at (1, 0) {$\jac_{dn,g}$};
    \node (A1_0) at (0, 1) {$(C,\shQ)$};
    \node (A1_1) at (1, 1) {$(C,\shQ^n)$};
    \path (A0_0) edge [->]node [auto] {$\scriptstyle{n}$} (A0_1);
    \path (A1_0) edge [|->,gray]node [auto] {$\scriptstyle{}$} (A1_1);
  \end{tikzpicture}
  \]
is the composition of a $\mu_n$-gerbe followed by a surjective cover and, if $p\nmid n$ or $g=0$, it is smooth.
\end{prop}
\begin{proof}
 Consider the diagram
   \[
  \begin{tikzpicture}[xscale=3.1,yscale=-1.2]
    \node (A0_0) at (0.3, -0.4) {$\jac_{d,g}$};
    \node (A0_1) at (1, 0) {$\stF$};
    \node (A0_2) at (2, 0) {$\jac_{dn,g}$};
    \node (A1_1) at (1, 1) {$\Picsh^d_{\stM_{g,1}/\stM_g}$};
    \node (A1_2) at (2, 1) {$\Picsh^{dn}_{\stM_{g,1}/\stM_g}$};
    \path (A0_0) edge [->]node [auto,swap] {$\scriptstyle{h}$} (A0_1);
    \path (A0_1) edge [->]node [auto] {$\scriptstyle{}$} (A0_2);
    \path (A0_2) edge [->]node [auto] {$\scriptstyle{}$} (A1_2);
    \path (A1_1) edge [->]node [auto] {$\scriptstyle{[n]}$} (A1_2);
    \path (A0_0) edge [->,bend right=40]node [auto] {$\scriptstyle{\Omega}$} (A0_2);
    \path (A0_1) edge [->]node [auto] {$\scriptstyle{}$} (A1_1);
  \end{tikzpicture}
  \]
where the square diagram is Cartesian. By \ref{multiplication by two on jacobians} the map $[n]$ and therefore the map $\stF\arr\jac_{dn,g}$ are covers of degree $n^{2g}$ and they are \'etale if $p\nmid n$ or $g=0$. It remains to prove that $h\colon \jac_{d,g}\arr \stF$ is a $\mu_n$ gerbe. 
An object of $\stF$ over a scheme $S$ is a triple
$(C,\chi,\shQ)$ where $(C,\shQ)\in \jac_{dn,g}(S)$, $\chi\in \Picsh^d_C(S)$ satisfying $[n]\chi=[\shQ]$. The map $h$ sends $(C,\shT)$ to $(C,[\shT],\shT^n)$. Let $\stY\arr T$ be the base change of $h$ along a map $T\arr \stF$ given by the data $(q\colon C\arr T,\chi,\shQ)$. The objects in $\stY(T)$ are pairs $(\shT,\lambda)$ where $\shT$ is a degree $d$ invertible sheaf on $C$ with $\chi=[\shT]$ and $\lambda\colon \shT^n\arr \shQ$ is an isomorphism.
We can assume $\chi=[\shT_0]$ and, since $[n]\chi=[\shQ]$ means $[\shT_0^n]=[\shQ]$ so that $\shT_0^n \simeq \shQ \otimes q^*\shR$ for some invertible sheaf $\shR$ on $T$ by \ref{cbs for elliptic curves}, that $\stY(T)\neq \emptyset$.
In this case $\stY(T)$ is isomorphic to the category of pairs $(\overline \shT, \mu)$ where $\overline \shT$ is a degree $0$ invertible sheaf with $[\overline \shT]=0$ in $\Picsh^0_C$ and $\mu\colon \overline \shT^n\arr \odi C$ is an isomorphism. It follows that $\stY\arr T$ is the base change of the map $\stU_{0,g,n}\arr (\Picsh^0_{\stM_{g,1}/\stM_g})[n]$ along the map $T\arr (\Picsh^0_{\stM_{g,1}/\stM_g})[n]$ given by $(C,[\odi C])$ and it is therefore a $\mu_n$-gerbe (see \ref{the zero case and mu n}).
\end{proof}

\begin{prop}\label{from two to one}
The stack $\stU_{d,g,n}$ is algebraic,  flat and of finite type over $\stM_g$. Moreover
the map 
  \[
  \begin{tikzpicture}[xscale=3.0,yscale=-0.6]
    \node (A0_0) at (0, 0) {$\stU_{d,g,n}$};
    \node (A0_1) at (1, 0) {$\stU_{dn,g,1}$};
    \node (A1_0) at (0, 1) {$(C,\shQ,s)$};
    \node (A1_1) at (1, 1) {$(C,\shQ^n,s)$};
    \path (A0_0) edge [->]node [auto] {$\scriptstyle{\Omega}$} (A0_1);
    \path (A1_0) edge [|->,gray]node [auto] {$\scriptstyle{}$} (A1_1);
  \end{tikzpicture}
  \]
is flat, surjective and of finite type. If $d>0$ then $\rho_{dn,g,1}\circ \Omega = \rho_{d,g,n}$ (see \ref{description of Udn and the cover Hdn} for the notation) and, in particular, $\Omega^{-1}(\stZ_{dn,g,1})=\stZ_{d,g,n}$.
If $p\nmid n$ or $g=0$ then $\Omega$ is smooth and $\stU_{d,g,n}$ is smooth over $\stM_g$.
\end{prop}
\begin{proof}
The equality $\rho_{dn,g,1}\circ \Omega = \rho_{d,g,n}$ follows from the fact that the base change of $\stH_{dn,g,1}\arr\stU_{dn,g,1}$ along $\Omega$ is $\stH_{d,g,n}\arr\stU_{d,g,n}$. Since the map $\Omega$ is the base change of $\jac_{d,g}\arr \jac_{dn,g}$ (see \ref{multiplication by n in Xd}) along the map $\stU_{dn,g,1}\arr \jac_{dn,g}$, by \ref{multiplication by n in Xd} we can reduce the problem to the case $n=1$, where all the claims follow from \ref{Ud1 has Hilbert scheme} and \ref{the zero case and mu n}.
\end{proof}

The remaining part of the section is dedicated to the proof of Theorem \ref{key proposition for the bad locus}. In particular in what follows we assume $d>0$.

\begin{proof}[Proof of theorem \ref{key proposition for the bad locus}, (\ref{key proposition for the bad locus: zero})]
 Set $\stH_{d,g,n}=\Spec \alA$. By definition $\stZ_{d,g,n}$ is the zero locus of a section of $(\det\alA)^{-2}$. We have an exact sequence on $\shC$
 \[
 0\arr \shL^{-n}\arr \odi \shC \arr \odi{\stH_{d,g,n}}\arr 0
 \]
 Applying $\pi_*$ we get an exact sequence on $\stU_{d,g,n}$
 \[
 \pi_*\shL^{-n}=0\arr \odi{\stU_{d,g,n}} \arr \alA \arr \R^1\pi_* (\shL^{-n})\arr \R^1\pi_*\odi \shC \arr 0=\R^1\pi_*\odi{\stH_{d,g,n}}
 \]
 where the first and last vanishing follows from \ref{cbs for elliptic curves}. We also have
 \[
 \R^1\pi_* (\shL^{-n}) \simeq \duale{\pi_*(\shL^n\otimes \omega_\pi)} \text{ and }
 \R^1\pi_*\odi \shC\simeq \duale{(\pi_*\omega_\pi)}
 \]
 again by \ref{cbs for elliptic curves}.
 We can now deduce the following formula for the determinant of $\alA$
 \[
 \det \alA \simeq (\det \pi_*(\shL^n\otimes\omega_\pi))^{-1}\otimes \det \pi_*\omega_\pi
 \]
 so that $\det \alA^{-2}$ is isomorphic to the invertible sheaf in the statement.
\end{proof}

\begin{remark}\label{lenght and flatness}
 If $R$ is a local ring, $R\arr S$ is a flat map of rings and $M$ is an $R$-module then
$$
  \l_S(M\otimes_R S)=\l_S(S/m_RS)\l_R(M)
$$
where $\l$ denotes the length function.
\end{remark}

The following Theorem is one of the crucial points of the whole paper.

\begin{theorem}\label{reducedness of the discriminant locus over Hilbert scheme}
 Assume $n\geq 2$. The discriminant locus of the universal degree $n$ cover of $\Hilb^n_{\stM_{g,1}/\stM_g}$ is flat, surjective and geometrically integral over $\shM_g$. In particular it is integral.
\end{theorem}
\begin{proof}
 The problem is local on $\stM_g$, therefore we can replace $\stM_g$ by a noetherian scheme $Y$ and $\stM_{g,1}$ by a genus $g$ curve $C$ over $Y$. 
 Set $Z$ for the discriminant locus. By definition $Z$ 
is the zero locus of a section of an invertible sheaf over $\Hilb^n_{C/Y}$. Moreover this section is always non zero over a geometric point of $Y$ because curves over an algebraically closed field always have $n$ distinct rational points. By \ref{cokernel of a geometrically injective map of flat sheaves} we can conclude that $Z$ is flat over $Y$. For the remaining part of the statement
we can assume $Y=\Spec k$, where $k$ is an algebraically closed field.

Let $C'$ be a non-empty smooth open subset of a projective and integral curve over $k$ (we will reduce to a calculation on a open subset of a plane curve).
Given indices $i\neq j\leq n$ we denote by $\Delta_{i,j}(C'^n)$ the effective Cartier divisor of ${C'}^n$ given by
$$
\Delta_{i,j}(C'^n)=\{(p_1,\dots,p_n)\in {C'}^n \st p_i=p_j \}
$$
Let $\shH(C'^{n+1})\subseteq {C'}^{n+1}$ be the zero locus of $1\in \odi{C'^{n+1}}(\Delta_{1,n+1}(C'^{n+1})+\cdots+\Delta_{n,n+1}(C'^{n+1}))$. 
Notice that if $C''\subseteq C'$ is a non-empty open subset then $\shH(C''^{n+1})\arr C''^{n}$ is the restriction of $\shH(C'^{n+1})\arr C'^{n}$. Indeed by definition we have
\begin{displaymath}
 \forall i\ \ \ \Delta_{i,n+1}(C'^{n+1}) \cap (C''^n \times C')=\Delta_{i,n+1}(C''^{n+1})\then \shH(C'^{n+1})\cap (C''^n \times C')=\shH(C''^{n+1})
\end{displaymath}
We claim that $\shH(C'^{n+1})\arr C'^n$ is a degree $n$ cover. Indeed, by discussion above, we can assume that $C'$ is projective: in this case the map $\shH(C'^{n+1})\arr C'^n$ is flat thanks to \ref{cokernel of a geometrically injective map of flat sheaves}, proper, quasi-finite and generically \'etale of degree $n$. Set $Z'$ for the discriminant locus of $\shH(C'^{n+1})$ in $C'^n$ and $P\in {C'}^n$ for the
generic point of $\Delta_{1,2}(C'^n)$, which lies in $Z'$. We will show that $\l(\odi{Z',P})=2$. We first show how to
conclude the proof from this fact. When $C'=C$ the cover $\shH(C^{n+1})$ induces a map
$$
 f\colon C^n\arr \Hilb^n_{C/k}
$$
which factors through an isomorphism $S^nC\simeq \Hilb^n_{C/k}$. In particular $f$ is a cover and $Z'=f^{-1}(Z)$.
Moreover, topologically, $Z=f(\Delta_{1,2}(C^n))$ and it coincides with the ramification locus of $f$.  In particular $Z$ is non empty and irreducible.
Since $\Hilb^n_{C/k}$ is smooth, to see that $Z$ is reduced we have to prove that $\l(\odi{Z,f(P)})=1$.
By \ref{lenght and flatness} and our assumption we have
$$
2=\l(\odi{Z',P})=\l(\odi{Z',P}/m_{f(P)}\odi{Z',P})\l(\odi{Z,f(P)})\geq \l(\odi{Z,f(P)})
$$
and we can exclude the equality because otherwise $f$ will be unramified in $P$, and by transitivity of $S_n$, also over $f(P)$, which is
not the case.

We have to prove that $\l(\odi{Z',P})=2$. Thanks to the above discussion we see that the ring $\odi{Z',P}$ (and therefore the number $\l(\odi{Z',P})$) does not change if we
take an open subset of $C'$. In particular, projecting in $\PP^2_k$, we can assume that $C'$ is of the form
$$
 C'=\Spec A \text{ where } A=k[x,y]_g/(f)
$$
and $f,g$ are polynomials.
We have ${C'}^n=\Spec B$, where
$$
B=A^{\otimes n}=k[x_1,y_1,\dots,x_n,y_n]_{g(x_1,y_1)\cdots g(x_n,y_n)}/(f(x_1,y_1),\dots,f(x_n,y_n))
$$
and $P=(x_2-x_1,y_2-y_1)$. By definition $\shH(C'^{n+1})$ is the spectrum of the $B$-algebra
$$
D=B[\alpha,\beta]_{g(\alpha,\beta)}/(f(\alpha,\beta),\prod_j I_j) \text{ where } I_j=(\alpha-x_j,\beta-y_j)
$$
We have to compute the discriminant locus of the cover $D_P=D\otimes B_P$ over $B_P$ and show that it has length $2$.
Since $C'$ is smooth, $B_P$ is a DVR. In particular we can assume that $P$ is generated by $x_1-x_2$ in $B_P$, so that
$x_1-x_2 \mid y_1-y_2$ in $B_P$.
Notice that
$$
(\alpha-x_1)\cdots(\alpha-x_l),(\alpha-x_{l+1})\in I_1\cdots I_l+I_{l+1}\then (x_{l+1}-x_1)\cdots(x_{l+1}-x_l)\in I_1\cdots I_l+I_{l+1}
$$
in $B[\alpha,\beta]_{g(\alpha,\beta)}/(f(\alpha,\beta))$. 
Looking at the quotient $B/P=A^{\otimes(n-1)}$ we see that $(x_{l+1}-x_1)\cdots(x_{l+1}-x_l)\notin P$ for $l>1$. So $I_1\cdots I_l+I_{l+1}$ is the
trivial ideal in $D_P$ and
applying the Chinese remainder theorem inductively it follows that
$$
D_P=(D_P/I_1 I_2)\times D_P/I_3 \times \cdots \times D_P/I_n
$$
Since $D_P/I_l\simeq B_P$, the discriminant locus of $D_P$ over $B_P$ coincides with the discriminant locus of $E=(D_P/I_1 I_2)$ over $B_P$, which is
a cover of degree $2$, since $D_P$ is a cover of degree $n$ of $B_P$. From $I_1I_2=0$ in $E$ we get relations
$$
\alpha^2=(x_1+x_2)\alpha - x_1x_2\text{ and } (x_1-x_2)\beta=(y_1-y_2)\alpha+x_1y_2-y_1x_2
$$
Since $x_1-x_2$ divides $y_1-y_2$ in $B_P$, $E$ is generated by $1,\alpha$ as $B_P$-module.
Moreover, since $E$ is a free $B_P$-module of rank $2$, $1,\alpha$ is also a $B_P$-basis of $E$. Finally a direct computation shows that
$$
\tr(\alpha)=x_1+x_2\comma \tr(\alpha^2)=(x_1+x_2)^2-2x_1x_2\comma \det \left( \begin{array}{cc}
\tr(1) & \tr(\alpha) \\
\tr(\alpha) & \tr(\alpha^2)
 \end{array} \right) = (x_1-x_2)^2
$$
where $\tr=\tr_{E/B_P}$. The last determinant is the discriminant section of the cover $E/B_P$ and therefore its discriminant locus has length $2$,
as claimed.
\end{proof}

\begin{proof}[Proof of theorem \ref{key proposition for the bad locus}, first sentence and (\ref{key proposition for the bad locus: first})]
 By \ref{from two to one} we have $\stZ_{d,g,n}=\emptyset \iff \stZ_{dn,g,1}=\emptyset$ and, by \ref{reducedness of the discriminant locus over Hilbert scheme}, this happens if and only if $dn=1$. So assume $dn>1$. Again by \ref{from two to one} we can assume $n=1$. The result then follows from \ref{Ud1 has Hilbert scheme} and \ref{reducedness of the discriminant locus over Hilbert scheme}.
\end{proof}

We now deal with the problem of reducibility of the stacks $\stZ_{d,g,n}$.

\begin{lemma}\label{strongly generated for elliptic curves}
 Let $k$ be an algebraically closed field and $C$ be a genus $g$ curve over $k$. If $\shQ$ is an invertible sheaf of degree $d$
 with $d> 2g-2$ and $q\in C(k)$ then the map
 \begin{align}\label{strongly generated map prop}
 \Hl^0(\shQ)\arr \shQ \otimes (\odi{C,q}/m_q^2\odi{C,q})
\end{align}
where $m_q$ is the maximal ideal of $\odi {C,q}$, has cokernel $\Hl^0(\shQ^{-1}\otimes \odi C(2q)\otimes \omega_C)$ and it is surjective if $d>2g$.
\end{lemma}
\begin{proof}
Consider the exact sequence
\begin{displaymath}
 0\arr \shQ\otimes\odi C(-2q)\arr \shQ\arr \shQ \otimes (\odi{C,q}/m_q^2\odi{C,q})\arr 0
\end{displaymath}
Since $\Hl^1(\shQ)\simeq \Hl^0(\shQ^{-1}\otimes \omega_C)=0$ by degree reasons, applying $\Hl^0$ we get an exact sequence
\begin{displaymath}
 \Hl^0(\shQ) \arr \shQ \otimes (\odi{C,q}/m_q^2\odi{C,q}) \arr \Hl^1(\shQ\otimes\odi C(-2q)) \simeq \Hl^0(\shQ^{-1}\otimes\odi C(2q)\otimes \omega_C)\arr 0
\end{displaymath}
Finally, if $d>2g$ then $\Hl^0(\shQ^{-1}\otimes\odi C(2q)\otimes \omega_C)=0$, again by degree reasons.
\end{proof}

\begin{proof}[Proof of theorem \ref{key proposition for the bad locus}, (\ref{key proposition for the bad locus: second})]
The case $n=1$ follows from \ref{Ud1 has Hilbert scheme}
and \ref{reducedness of the discriminant locus over Hilbert scheme}. So we focus on the case $nd>2g-2$. We remark that the proof is a bit simpler if we have the stronger inequality $nd>2g$, and the intermediate cases require a finer inspection.

Set $\stV=\stV_{d,g,n}$ and $\stU=\stU_{d,g,n}$. By \ref{when Vdn is a vector bundle} $\stV$ is a vector bundle over $\jac_{d,g}$
and $\stU$ is the complement of the zero section. Consider the diagram
  \[
  \begin{tikzpicture}[xscale=1.7,yscale=-1.0]
    \node (A0_0) at (0, 0) {$\stV \times_{\jac_{d,g}} \stC$};
    \node (A0_1) at (1, 0) {$\shC$};
    \node (A1_0) at (0, 1) {$\stV$};
    \node (A1_1) at (1, 1) {$\jac_{d,g}$};
    \path (A0_0) edge [->]node [auto] {$\scriptstyle{}$} (A0_1);
    \path (A1_0) edge [->]node [auto] {$\scriptstyle{}$} (A1_1);
    \path (A0_1) edge [->]node [auto] {$\scriptstyle{\pi}$} (A1_1);
    \path (A0_0) edge [->]node [auto] {$\scriptstyle{}$} (A1_0);
  \end{tikzpicture}
  \]
where $\shC$ is the universal curve over $\jac_{d,g}$. Denote by $\shL$ the universal invertible sheaf over $\shC$, so that $\stV$ corresponds
to $\pi_*\shL^n$.
From \cite[Section 16.7]{EGAIV-4} there exists a locally free sheaf $\shF$ on $\shC$, the $2$-th bundle of principal parts of $\shL^n$, and a map
$\alpha\colon \pi^*\pi_*(\shL^n)\arr \shF$
such that for all algebraically closed fields $k$ and triples $(C,\shQ,q)\in \shC(k)$, where $q\in C(k)$, we have 
$\shF\otimes k \simeq \shQ^n \otimes \odi {C,q}/m_q^2$, where $m_q$ is the maximal ideal of $\odi {C,q}$ and
$$\label{strongly generated map}
\alpha\otimes k \colon \Hl^0(\shQ^n) \simeq \pi^*\pi_*(\shL^n) \otimes k \arr \shF\otimes k \simeq \shQ^n \otimes \odi {C,q}/m_q^2
$$
is the restriction. If $nd>2g$, by \ref{strongly generated for elliptic curves} we can conclude that $\alpha$ is surjective. In this case in what follows set $\stW=\jac_{d,g}$ and $\stY=\emptyset$. If $nd\leq 2g$ we want to find an open substack $\stW$ of $\jac_{d,g}$ over which $\alpha$ is surjective. If $dn=2g$ and $g\geq 3$ consider the map
\begin{displaymath}
\beta\colon \stM_{g,1}\arr \jac_{dn,g}\comma (C\arr S,\sigma)\longmapsto (C,\odi C(2\sigma)\otimes \omega_{C/S})
\end{displaymath}
while if $dn=2g-1$ and $g\geq 4$ the map
\begin{displaymath}
 \beta\colon \stM_{g,1} \times_{\stM_g} \stM_{g,1}\arr \jac_{dn,g}\comma (C\arr S,\sigma,\tau) \longmapsto (C,\odi C(2\sigma-\tau)\otimes \omega_{C/S})
\end{displaymath}
In both cases denote by $\stY'$ the closed substack of $\jac_{dn,g}$ whose topological space is the closure of the image of $\beta$, $\stY=n^{-1}(\stY')$ where $n\colon \jac_{d,g} \arr \jac_{dn,g}$ is the elevation to the $n$-th power and $\stW=\jac_{d,g}-\stY$.

 We will denote by $-_{\stW}$ the restriction to $\stW$. We want to prove that $\alpha_\stW$ is surjective and we can check this on the geometric points of $\shC_\stW$. Given $(C,\shQ,q)\in \shC_\stW(k)$, where $k$ is an algebraically closed field, by \ref{strongly generated for elliptic curves} the cokernel of $\alpha\otimes k$ is $\Hl^0(\shT)$ where $\shT=\shQ^{-n}\otimes \odi C(2q)\otimes \omega_C$. Notice that $\shT$ has degree $-dn+2g$. Assume by contradiction $\Hl^0(\shT)\neq 0$. If $dn=2g$ then $\shT\simeq \odi C$ and therefore $n(C,\shQ)\simeq \beta(C,q)$ which is not the case by construction of $\stW$. Finally if $dn=2g-1$ then by \ref{degree one sheaves and sections} there exists $q'\in C$ such that $\shT\simeq \odi C(q')$, which means $n(C,\shQ)\simeq \beta(C,q,q'
)$ and it is again not possible by construction.

We want to prove that $(\stZ_{d,g,n})_\stW$ is irreducible. Since $\jac_{d,g}$ and $\shC$ are integral and the vector bundle associated with $\pi^*\pi_*(\shL^n)$ is $\stV \times_{\jac_{d,g}} \stC$, the kernel $\stZ$ of $\alpha_\stW$ is an integral closed substack of $(\stV\times_{\jac_{d,g}}\shC)_\stW$.
Let $\widehat \stZ$ be the image of $\stZ$ via the projection $(\stV\times_{\jac_{d,g}}\shC)_\stW \arr \stV_\stW$. We want to prove that 
$(\stZ_{d,g,n})_\stW=\widehat \stZ \cap \stU$ topologically. This will imply the irriducibility of $(\stZ_{d,g,n})_\stW$.

In what follows $k$ will be an arbitrary algebraically closed field. The objects of $\stZ(k)$ are tuples $(C,\shQ,q,s)\in (\stV\times_{\jac_{d,g}}\shC)_\stW (k)$, where $q\in C(k)$ and $s\in \shQ^n$, such that $v_q(s)\geq 2$, 
where $v_q$ denotes the valuation in $q$. Thus the objects of $(\widehat \stZ \cap \stU)(k)$ are triples $(C,\shQ,s)\in \stU_\stW(k)$ 
for which there exists $q\in C(k)$ such that $v_q(s)\geq 2$. The result then follows from the following remark. If $(C,Q,s)\in \stU(k)$ then
the zero locus $Z(s) \subseteq C$ of $s$ is \'etale over $k$ if and only if for all $q\in C(k)$ the ring $\odi {C,q}/s_q( \simeq \odi {Z(s),q})$ is either
$0$ or $k$, that is if and only if for all $q\in C(k)$ we have $v_q(s) < 2$.

If $dn>2g$, so that $\stW=\jac_{d,g}$, then $(\stZ_{d,g,n})_\stW=\stZ_{d,g,n}$ is irreducible as required. So assume $2g-2<dn\leq 2g$ and $g\geq 2$.
Denote by $f\colon \stU\arr \jac_{d,g}$ the structure map. Topologically we have
\begin{displaymath}
 \stZ_{d,g,n}\subseteq \overline{(\stZ_{d,g,n})_\stW} \cup f^{-1}(\stY)
\end{displaymath}
where the closure is taken inside $\stU$. If $\stZ_{d,g,n}=\stU$ (which a posteriori will not be the case) there is nothing to prove. Otherwise, since $\stZ_{d,g,n}\neq \emptyset$, by \ref{dimension of Cartier divisors for stacks} the equality $\stZ_{d,g,n} = \overline{(\stZ_{d,g,n})_\stW}$ follows from $\dim \stU - \dim f^{-1}(\stY)\geq 2$, which we are going to prove. Thanks to \ref{multiplication by n in Xd} the map $\stU\arrdi f \jac_{d,g} \arrdi n \jac_{dn,g}$ has constant relative dimension. Moreover $n^{-1}(\stY')=\stY$ and, by \ref{dimension formulas for stacks}, we get
\begin{displaymath}
 \dim \stU - \dim f^{-1}(\stY)=\dim \jac_{dn,g}-\dim \stY'
\end{displaymath}
If $dn=2g$ set $\stM=\stM_{g,1}$, while if $dn=2g-1$ set $\stM=\stM_{g,1} \times_{\stM_g} \stM_{g,1}$. Let us also write $\delta\colon \jac_{dn,g}\arr \Picsh^{dn}_{\stM_{g,1}/\stM_g}$ for the natural map, and $\gamma=\delta\circ \beta$. Since $\stM$ and $\Picsh^{dn}_{\stM_{g,1}/\stM_g}$ are proper over $\stM_g$ and they are Deligne-Mumford stacks we can conclude that $\gamma(\stM)$ is closed and, by \ref{dimension of image for stacks}, that
\begin{displaymath}
 \dim \gamma(\stM) \leq \dim \stM
\end{displaymath}
Since $\delta \colon \jac_{dn,g}\arr \Picsh^{dn}_{\stM_{g,1}/\stM_g}$ is a $\Gm$-gerbe we also have $\delta^{-1}(\gamma(\stM))=\stY'$, so that, by \ref{dimension formulas for stacks},
\begin{displaymath}
 \dim \jac_{dn,g} - \dim \stY'= \dim \Picsh^{dn}_{\stM_{g,1}/\stM_g} - \dim \gamma(\stM) \geq g+\dim \stM_g - \dim \stM
\end{displaymath}
Since $\dim \stM=1+\dim \stM_g$ if $dn=2g$ and $\dim \stM=2+\dim \stM_g$ if $dn=2g-1$ we get the desired formula.
\end{proof}

When $dn\leq 2g-2$ it is not clear whether $\stZ_{d,g,n}$ is irreducible or not. The main technical issue here is that $\stV_{d,g,n}\arr \jac_{d,g}$ is no longer a vector bundle and $\stU_{d,g,n}\arr \jac_{d,g}$ may not be surjective.

When $dn > 2g-2$ we have shown above that the stack $\stZ_{d,g,n}$ is irreducible except in the following cases: $g=1$ and $(d,n)=(1,2)$; $g=2$ and $(d,n)\in \{(1,4),(2,2),(1,3)\}$; $g=3$ and $(d,n)=(1,5)$. In the remaining part of the section we work out the case $g=1$, $d=1$ and $n=2$ (so that $dn=2g$). As claimed in \ref{key proposition for the bad locus}, (\ref{key proposition for the bad locus: third}) we will see that $\stZ_{1,1,2}$ is reducible. In the other cases just listed, it is again not clear whether $\stZ_{d,g,n}$ is irreducible or not.

\begin{lemma}\label{cover for one of two times a section}
 Assume $p \neq 2$.
 Let $\shC\arr\stS$ be a curve over an algebraic stack with a section $\tau$ and denote by $\shW$ the zero locus of
 $1\in \odi \shC(2\tau)$ in $\shC$. Then $\rho\colon \stW\arr \stS$ is a degree $2$ cover. Moreover $\tau$ factors trough $\stW$
 and the induced map $\rho_*\odi \stW\arr\odi \stS$ is $\tr_{\rho_*\odi \stW}/2$, where $\tr$ denotes the trace map, and its kernel is a square zero ideal.
\end{lemma}
\begin{proof}
 The map $\rho \colon \stW\arr \stS$ is a degree $2$ cover thanks to \ref{degree one sheaves and sections} and the section $\stS\arr \shC$ factors through $\stW$ by definition of this last space.
 Denote by $\psi\colon \rho_*\odi \stW\arr \odi \stS$ the induced map. Since $\stM_{g,1}$ is reduced, we can assume $\stS=\Spec k$, where $k$ is an algebraically closed field.
 In this case the result follows because 
 $\rho_*\odi \stW = k[x]/(x^2)$.
\end{proof}

\begin{proof}[Proof of theorem \ref{key proposition for the bad locus}, (\ref{key proposition for the bad locus: third})]
If we denote by $\overline \E$ the universal curve over $\stM_{1,1}$, we are going to show that there is a $\mu_2$-gerbe
$\stZ_{1,1,2}\arr \overline \E[2]$. Since this last group is a disjoint union of two irreducible components and gerbes are
geometrically irreducible, we will conclude that $\stZ_{1,1,2}$ is also a disjoint union of two irreducible components. We will then
study the component over the zero section of $\overline \E[2]$ and represent it as zero locus of a section of an invertible sheaf.

 Set $\E=\shC$ for the universal curve of $\stU_{1,1,2}$.
 The canonical cover $ \stH_{1,1,2}\arr\stU_{1,1,2}$ of $\stU_{1,1,2}$ has degree $2$. By standard theory of double covers it is given by an invertible sheaf $\shT$ over $\stU_{1,1,2}$ and a section $r\in \shT^2$, so that $\stH_{1,1,2}=\Spec \alA$ where $\alA=\odi{\stU_{1,1,2}} \oplus \shT^{-1}$.
 By an easy local computation, the discriminant section of the canonical cover coincides, up to
 an invertible element, with $r$ and therefore $\stZ_{1,1,2}$ is the zero locus of $r$. Set $\stZ=\stZ_{1,1,2}$. 
 We will use the symbol $-_\stZ$ for
 base changes along $\stZ\arr \stU_{1,1,2}$. For instance $\E_\stZ$ is the universal curve over $\stZ$ with universal invertible sheaf $\shL_\stZ$.
 Since $r_\stZ=0$ we have $(\shT_\stZ^{-1})^2=0$
 in $\alA_\stZ$. Therefore the projection $\alA_\stZ \arr \odi \stZ$ is a ring homomorphism
 and thus induces a section $\stZ \arr \stH_\stZ$ and therefore a section
 $\tau\colon \stZ \arr \E_\stZ$. This yields a unique map $\odi {\E_\stZ}(\tau)\arr \shL_\stZ^2$ 
 that sends $1$ to $s$ and therefore a section $s'\in \shN=\shL_\stZ^2\otimes\odi {\E_\stZ}(-\tau)$.
 Since $s'$ is non zero on the geometric fibers of $\pi_\stZ \colon \E_\stZ\arr \stZ$, by \ref{degree one sheaves and sections}
 there exists another section $\tau'\colon \stZ \arr \E_\stZ$ and an isomorphism $\odi {\E_\stZ}(\tau')\simeq \shN$
 sending $1$ to $s'$. Since the cover $(\stH_{1,1,2})_\stZ\arr \stZ$ is topologically an homeomorphism,
 the sections $\tau$ and $\tau'$ coincide on the geometric fibers of $\pi_\stZ$. Since
 $\stZ$ is reduced thanks to \ref{key proposition for the bad locus}, (\ref{key proposition for the bad locus: first}), we can conclude that
 $\tau=\tau'$.
 Moreover the induced isomorphism $\odi {\E_\stZ}(2\tau)\arr \shL_\stZ^2$ sends $1$ to $s$.
 
 Define $\stZ'$ as the stack of tuples $(E,\shG,\tilde \tau,\lambda)$ where $E$ is a genus one curve over $S$, $\shG$ is a degree $0$
 invertible sheaf over $E$, $\tilde \tau$ is a section of $E$ and $\lambda\colon \shG^2\arr \odi E$ is an isomorphism.
 Discussion above shows that we have a map $\stZ\arr \stZ'$ which sends $(E,\shQ,s)$ to $(E,\shG,\tilde \tau,\lambda)$
 where $\tilde \tau$ is induced by the section $\tau\colon \stZ\arr \E_\stZ$, $\shG=\shQ\otimes \odi E(-\tilde  \tau)$
 and the isomorphism $\lambda$
 is the base change of the isomorphism $\odi {\E_\stZ}(2\tau)\arr \shL_\stZ^2$.
 Conversely we can define a map $\stZ'\arr \stZ$ by
 sending $(E,\shG,\tilde \tau,\lambda)$ to $(E,\shQ,s)$ where $\shQ=\shG\otimes \odi E(\tilde \tau)$ and $s$ 
 is the image of $1$ under the isomorphism
 \begin{displaymath}
  \odi E(2\tilde \tau)\simeq \shG^2 \otimes \odi E(2\tilde \tau) \simeq \shQ^2
 \end{displaymath}
 By \ref{cover for one of two times a section} we see that the last functor is well defined and
 that the composition $\stZ'\arr \stZ\arr \stZ'$ is equivalent to the identity.
 Conversely the composition $\stZ\arr\stZ'\arr\stZ$ is equivalent to the identity
because the map $\odi {\E_\stZ}(2\tau)\arr \shL_\stZ^2$ sends $1$ to $s$.

 We define the map
   \[
  \begin{tikzpicture}[xscale=4.0,yscale=-0.6]
    \node (A0_0) at (0, 0) {$\stZ'$};
    \node (A0_1) at (1, 0) {$\overline \E[2]$};
    \node (A1_0) at (0, 1) {$(E,\shG,\tilde\tau,\lambda)$};
    \node (A1_1) at (1, 1) {$(E,\tilde \tau,[\shG])$};
    \path (A0_0) edge [->]node [auto] {$\scriptstyle{p}$} (A0_1);
    \path (A1_0) edge [|->,gray]node [auto] {$\scriptstyle{}$} (A1_1);
  \end{tikzpicture}
  \]
where we identify $\overline \E$ with $\Picsh^0_{\overline \E/\stM_{1,1}}$ (see \ref{elliptic curves and pic}), which is easily seen to be a $\mu_2$-gerbe.

Now we prove that $\overline \E[2]$ is a disjoint union of two irreducible components, one being the zero section $\stM_{1,1}\arr \overline \E[2]$. 
First of all, since
$\overline \E[2]$ is \'etale, the zero section is a connected component of $\overline \E[2]$.
So we need to prove that the complement $H$ is irreducible as well.
But $H\arr \stM_{1,1}$ is an \'etale degree $3$ cover and thus, if $H$ is not connected (and therefore irreducible being smooth), it should have a section,
which is not the case thanks to \ref{torsion groups are autodual} and \ref{torsion elliptic has no sections}.

Since $p\colon\stZ'\arr \overline \E[2]$ is a $\mu_2$-gerbe and thus has irreducible fibers, we can conclude that $\stZ'\simeq \stZ$
is a disjoint union of two irreducible substack,
one of which is $\stZ_0=p^{-1}(\stM_{1,1})$. 
We identify $\stZ'$ with $\stZ$ and we are going to write $\stZ_0$ as the zero locus of a section of the invertible sheaf in the statement.
So with an object $(E,\shQ,s)\in \stZ$ are associated a section $\tilde \tau$ of $E$, base change of $\tau\colon \stZ\arr \E_\stZ$ and an
isomorphism $\odi E(2\tilde \tau)\arr \shQ^2$ sending $1$ to $s$, base change of the isomorphism $\odi {\E_\stZ}(2\tau)\arr \shL_\stZ^2$.
The objects of $\stZ_0$ are the triples $(E,\shQ,s)\in \stZ(S)$
such that $\shQ$ and $\odi E(\tilde \tau)$ differ by an
invertible sheaf from the base $S$, that is $[\shQ]=[\odi E(\tilde \tau)]$ in $\Picsh^1_{E/S}$. 
Since $\shL$ is an invertible sheaf of degree $1$ on $\E$, $\shR=\pi_*\shL$ is an invertible sheaf by \ref{cbs for elliptic curves}
and there exists a unique section $\sigma\colon \stU_{1,1,2}\arr \E$ with an isomorphism $\shL\simeq \odi \E(\sigma)\otimes \pi^*\shR$ by
\ref{elliptic curves and pic}.

Let $\shW$ be the zero locus in $\E$ of the section $1\in \odi \E(2\sigma)$. The induced map $\shW\arr \stU_{1,1,2}$ is a degree $2$ cover
by \ref{cover for one of two times a section}.
Tensoring the exact sequence defining $\odi \shW$ by $\shL^2$, we get an exact sequence
\[
0\arr \pi^*\shR^2\arrdi \alpha \pi^*\shR^2\otimes \odi E(2\sigma) \simeq \shL^2\arr \odi W \otimes \shL^2 \arr 0
\]
where $\alpha(x)=x\otimes 1$. Applying $\pi_*$ and taking into account
\ref{cbs for elliptic curves} we get an exact sequence 
\[
0\arr \shR^2\arrdi{ \pi_* \alpha} \pi_*\shL^2\arr \pi_*(\odi W \otimes \shL^2) \arr \R^1\pi_*(\pi^*\shR^2) \arr 0
\]
of locally free sheaves on $\stU_{1,1,2}$. Note that the exact sequence on $\E$ satisfies base change for $\pi_*$.
Set $\shN=\Coker(\pi_*\alpha)$. This is an invertible sheaf and applying the determinant we see that it coincides with the
invertible sheaf in the statement. The section $s\in \pi_*\shL^2$ induces a section $t\in \shN$ and we claim that its zero locus
is exactly $\stZ_0$. This will conclude the proof.
Let $\chi=(E\arrdi f S,\shQ,s)\in \stU_{1,1,2}$. We will denote by $-_\chi$ the base change along the corresponding map $S\arr\stU_{1,1,2}$.
For instance $\shL_\chi=\shQ$, $f=\pi_\chi$ and, with abuse of notation, $s_\chi=s$.
We have that $t_\chi=0$ if and only if $s_\chi\in \Imm(f_*\alpha_\chi)\subseteq f_*\shQ^2$ if and only if 
$s\in \Imm(\alpha_\chi)\subseteq \shQ^2$. So $t_\chi=0$ if and only if the square of the isomorphism
$\shQ\simeq f^*\shR_\chi\otimes \odi E(\sigma_\chi) $ sends $s$ to a section of the form $x\otimes 1$.
We want to show that those are exactly the objects of $\stZ_0$, that is $t_\chi=0$
if and only if $\chi\in \stZ_0$.

If $\chi\in \stZ_0$, we have $[\shQ]=[\odi E(\sigma_\chi)]=[\odi E(\tilde \tau)]$ in $\Picsh^1_{E/S}$, which implies $\sigma_\chi=\tilde \tau$
by \ref{elliptic curves and pic}. Moreover we have an isomorphism $\odi E(2\tilde \tau)\simeq \shQ^2$ sending $1$ to $s$.
We can conclude observing that all the isomorphisms $\odi E(2\tilde \tau)\arr \odi E(2\tilde \tau)\otimes f^*\shR^2$
send $1$ to a section of the form $1\otimes x$.

Assume now $t_\chi=0$, so that $s\in \shQ^2$ corresponds to a section $x\otimes 1\in f^*\shR_\chi^2 \otimes \odi E(2\sigma_\chi)$. If
$x\in f^*\shR_\chi^2$ does not generate this sheaf, then the zero locus of $s\in \shQ^2$ inside $E$ cannot be a cover of $S$, because
it will have non zero dimensional fibers, contradicting the fact that $\chi\in \stU_{1,1,2}$. So $f^*\shR_\chi^2\simeq \odi E$, and the
zero locus of $s$ in $E$ is the base change of $\shW\subseteq \E$, the zero locus of $1\in \odi \E(2\sigma)$. Taking into account
\ref{cover for one of two times a section}, this shows that
$\chi\in \stZ$. It also implies that $\tilde \tau=\sigma$, so that $[\shQ\otimes\odi E(-\tilde \tau)]=0$ in $\Picsh^0_{E/S}$. This exactly means
that $\chi\in \stZ_0$, as required.
\end{proof}

\section{Stacks of uniform cyclic covers and their Picard groups.}\label{sec: uniform covers}

In this section we work over a field of characteristic $p\geq 0$ and we fix a non negative integer $g$ and a positive integer $n$ with $n\geq 2$.

\begin{definit}\label{def: uniform cyclic covers}
 Let $Y$ be a scheme. A uniform cyclic cover of degree $n$ of $Y$ is a map $f\colon X\arr Y$ together with an action of $\mu_n$ on $X$ such that
 for all $q\in Y$ there exists an affine open neighborhood $U=\Spec R$ of $q$, an element $h\in R$ and a $\mu_n$-equivariant isomorphism of
  $U$-schemes $f^{-1}(U) \simeq \Spec R[x]/(x^n-h)$, where the right hand side is given the action for which $\deg x=1$.
  
 Uniform cyclic covers of degree $n$ form a stack that we denote by $\shU\shC_n$.
\end{definit}

Notice that uniform cyclic covers of degree $n$ are covers of degree $n$ and can be seen as a generalization of double covers when $p\neq 2$.
The definition of uniform cyclic covers in \cite{Arsie2004} is slightly different from our, because in \ref{def: uniform cyclic covers} we do not require that $h$ is a non zero divisor. The reason is that this is automatic for uniform cyclic covers between schemes smooth on a common base and that, avoiding this restriction, uniform cyclic covers are stable by base change.

\begin{definit}
 Let $h$ be a natural number. We denote by $\stB_{h,g,n}$ the stack of triples $(D,C,f)$ where
 $D\arr S$ 
 is a genus $h$ curve, $C\arr S$ is a genus $g$ curve and $f\colon D\arr C$ is a uniform cyclic cover of degree $n$.
 
 We define the number $d(h,g,n)=2\frac{h+n(1-g)-1}{n(n-1)}$, so that $h=1+n(g-1)+\frac{n(n-1)}{2}d(h,g,n)$.
\end{definit}

The aim of this section is to describe $\stB_{h,g,n}$ and compute its Picard group, at least for $h \gg g$. We start by describing explicitly uniform cyclic covers.

\begin{remark}\label{description of uniform cyclic covers}
 Let $\stY_n$ be the stack parametrizing pairs $(\shL,s)$ where $\shL$ is an invertible sheaf and $s\in \shL^n$. There is
 an equivalence $\stY_n\arr \stU\stC_n$ that maps $(\shL,s)\in \stY_n(S)$ to
 \begin{displaymath}
  X=\Spec \alA \arr S\text{ where }\alA=\odi S \oplus \shL^{-1} \oplus \cdots \oplus \shL^{-(n-1)}
 \end{displaymath}
where $\mu_n$ acts on $\alA$ via the given grading and the equivariant algebra structure on $\alA$ is obtained as follows:
given $0\leq u,v,z<n$ such that $z\equiv u+n \text{ mod }(n)$ the multiplication is
\begin{displaymath}
 (\shL^{-u}\otimes \shL^{-v}\arr \shL^{-z}) \simeq \begin{cases}
                                             \shL^{-u-v}\arrdi \id \shL^{-z}&\text{ if } u+v<n\\
                                             \shL^{-u-v}\simeq \shL^{-z}\otimes \shL^{-n} \arrdi{\id\otimes s} \shL^{-z}&\text{ if } u+v\geq n
                                            \end{cases}
\end{displaymath}
A quasi inverse $\Lambda\colon \stU\stC_n \arr \stY_n$ is obtained as follows. Given a uniform cyclic cover $f\colon X\arr S$ of degree $n$, the
group $\mu_n$ acts on $f_*\odi X$. The degree $1$ part of $f_*\odi X$ is an invertible sheaf on $S$ and we set $\shL$ for its dual. Since the
multiplication $f_*\odi X\otimes f_*\odi X\arr f_*\odi X$ is $\mu_n$-equivariant, we get a map from $\shL^{-n}$ to the degree $0$ part of
$f_*\odi X$, which is $\odi S$. This yields a section $s\in \shL^n$.
\end{remark}

\begin{prop}\label{relation between Bhg and Udtwo}
 Let $h$ be a natural number and set $d=d(h,g,n)$. If $d\notin \N$ then $\stB_{h,g,n}=\emptyset$. If $d\in \N$ the functor (see 
 \ref{description of uniform cyclic covers} for the notation)
   \[
  \begin{tikzpicture}[xscale=3.1,yscale=-0.6]
    \node (A0_0) at (0, 0) {$\stB_{h,g,n}$};
    \node (A0_1) at (1, 0) {$\stU_{d,g,n}$};
    \node (A1_0) at (0, 1) {$(D,C,f)$};
    \node (A1_1) at (1, 1) {$(C,\Lambda_C(f))$};
    \path (A0_0) edge [->]node [auto] {$\scriptstyle{\psi_{h,g,n}}$} (A0_1);
    \path (A1_0) edge [|->,gray]node [auto] {$\scriptstyle{}$} (A1_1);
  \end{tikzpicture}
  \]
is well defined and an open immersion. If $h>n(g-1)+1$ the image of $\psi_{h,g,n}$ is the complement of $\stZ_{d,g,n}$ in $\stU_{d,g,n}$, which is the \'etale locus of the canonical cover $\stH_{d,g,n}\arr \stU_{d,g,n}$. If $h=n(g-1)+1$ and $g\geq 1$ the image of $\psi_{h,g,n}$
is the substack of $\stU_{0,g,n}$ of triples $(C\arr S,\shQ,s)$ such that $\shQ,\dots,\shQ^{n-1}$ are not trivial on the geometric fibers of $C\arr S$.
\end{prop}
\begin{proof}
 Let $(D,C,f)\in \stB_{h,g,n}(k)$, where $k$ is an algebraically closed field,
 and set $(\shQ,s)=\Lambda_C(f)$, so that 
 $$
 f_*\odi D \simeq \odi C \oplus \shQ^{-1} \oplus \cdots \oplus \shQ^{-(n-1)}
 $$ 
 Since $D$ is integral and connected we get $\dim_k \Hl^0(\shQ^{-i})=0$ for $i=1,\dots,n-1$. By Riemann-Roch
 it follows that $\dim_k \Hl^1(\shQ^{-i})=i\deg \shQ + g-1$ and therefore
 \begin{displaymath}
  \dim_k \Hl^1(f_*\odi D)=h=n(g-1)+1+(\deg \shQ )n(n-1)/2
 \end{displaymath}
 In particular $\deg \shQ = \in \Z$.
 On the other hand $s$ cannot be zero since $D$ is smooth and therefore
 $\deg \shQ\geq 0$. 
 In conclusion we see that $\stB_{h,g,n}=\emptyset$ if $d\notin \N$, and that $\psi_{h,g,n}$ is well defined if $d\in \N$.
 
 From now on we assume $d\in \N$. From \ref{description of uniform cyclic covers} it follows that $\psi_{h,g,n}$ is fully faithful.
 
 Given $(C\arr S,\shQ,s)\in \stU_{d,g,n}$ we have to check under what conditions
 the total space of the unifor cyclic cover $D\arr C$ associated with $(\shQ,s)$ (see \ref{description of uniform cyclic covers}) is a smooth
 curve of genus $h$ over $S$.
 It is easy to see that everything follows from the case $S=\Spec k$, where $k$ is an algebraically closed field. 
 Assume $d>0$. We have $\Hl^0(\shQ^{-i})=0$ for $i>0$, which tells us that $D$ is connected and,
 by definition of $d$, that $\dim_k \Hl^1(\odi D)=h$. The result then follows because
 the scheme $D$ is
 regular if and only if the zero locus of $s\in \shQ^n$ is \'etale over $k$. This can be checked locally using that if $(R,m_R)$ is a DVR and $h\in R$
 then $R[x]/(x^n-h)$ is regular if and only if $h\notin m_R^2$.
 
 Now assume $d=0$. By \ref{the zero case and mu n} the map $s\colon \odi C \arr \shQ^n$ is an isomorphism.
 In particular $D\arr C$ is a $\mu_n$-torsor and therefore $D$ is smooth. Moreover $D$ is connected if and only if $\Hl^0(\shQ^{-i})=0$ for
 $i=1,\dots,n-1$, in which case has exactly genus $h$ by definition of $d$.
 Since $\shQ$ has degree $0$ we have that $\Hl^0(\shQ^{-i})\neq 0$ if and only if $\shQ^i\simeq \odi C$, which
 concludes the proof.
\end{proof}

\begin{prop}\label{Bhgr not empty if d notin N}
 Let $h$ be a natural number with $d=d(h,g,n)\in \N$. Then $\stB_{h,g,n}$ is a non empty algebraic stack of finite type and if $nd>2g-2$ or $p\nmid n$ then $\stB_{h,g,n}\arr \stM_g$ is smooth and surjective.
\end{prop}
\begin{proof}
Let $k$ be an algebraically closed field. We are going to prove that $\stB_{h,g,n}\neq \emptyset$ and, if $dn>2g-2$ or $p\nmid n$, that $\stB_{h,g,n}\arr \stM_g$ is surjective. All the other claims follow from \ref{when Vdn is a vector bundle}, \ref{from two to one} and \ref{relation between Bhg and Udtwo}

Assume $d>0$. By \ref{relation between Bhg and Udtwo} we have $\stB_{h,g,n}=\stU_{d,g,n}-\stZ_{d,g,n}$. Moreover by \ref{from two to one} there is a surjective map $\stU_{d,g,n}-\stZ_{d,g,n}\arr \stU_{nd,g,1}-\stZ_{nd,g,1}$. We can conclude that $\stB_{h,g,n}\arr \stM_g$ is surjective because if $C$ is a  genus $g$ curve over $k$ and $p_1,\dots,p_{nd}$ are distinct
 rational points then $(C,\odi C(p_1+\dots+p_{nd}),1)\in (\stU_{nd,g,1}-\stZ_{nd,g,1})(k)$.
 
 Assume $d=0$ and let $C$ be a genus $g$ curve over $k$. We have $g\geq 1$ because if $g=0$ then $h=1-n<0$. By \ref{relation between Bhg and Udtwo} the fiber of $\stB_{h,g,n}\arr \stM_g$ over $C\in \stM_g(k)$ is not empty if and only if $\Pic C$ has an element of order $n$. If $p\nmid n$ this is always the case thanks to \ref{multiplication by two on jacobians}. If $p\mid n$ we have to show that this holds when $C$ is general. We can assume $n=q^l$ for some prime $q$. If $q\neq p$ then $\Pic C$ has an element of order $n$ by \ref{multiplication by two on jacobians}. Assume $p=q$. By \cite[Theorem 2.3]{Faber2004} when $C$ is general there exists an invertible sheaf
 $\shT$ on $\shC$ of order $p$. Since $[p^{l-1}]\colon \Picsh^0_C \arr \Picsh^0_C$ is surjective by
 \ref{multiplication by two on jacobians}, there exists $\shQ\in \Pic C$ such that $\shQ^{p^{l-1}}\simeq \shT$. It is easy to check
 that $\shQ$ has order exactly $p^l=n$.
\end{proof}

The following result explains the relation between $\Pic \stB_{h,g,n}$ and $\Pic \jac_{d(h,g,n),g}$.

\begin{prop}\label{general prop for main theorem}
 Let $h$ be a natural number such that $d=d(h,g,n) \in \Z$ and $nd>2g-2$ and
 let $\pi\colon \shC \arr \jac_{d,g}$ be the universal curve 
 and $\shL$ be the universal invertible sheaf on $\shC$. Set also 
 $\shT=(\det\pi_*(\shL^n \otimes \omega_\pi))^2 \otimes (\det \pi_*\omega_\pi)^{-2}$.
 Then the map $\shB_{h,g,n}\arr \jac_{d,g}$ induces a surjective morphism
 \begin{displaymath}
  \gamma\colon (\Pic \jac_{d,g})/\langle \shT \rangle \arr \Pic \stB_{h,g,n}
 \end{displaymath}
 If $\stZ_{d,g,n}$ is integral then $\gamma$ is an isomorphism. If $h=n=2$, $g=1$ (so that $d=1$) and $p\neq 2$ then
 the kernel of $\gamma$ is generated by $(\pi_*\shL)^2\otimes (\pi_*\omega_\pi)^{-2}$.
 
 If $g=0$, with notations from \ref{picard jacobian genus zero}, then $\shT \simeq \shL_0^{2n(nd-1)}$ if $d$ is even, 
 $\shT\simeq \shL_0^{n(nd-1)}$ if $d$ is odd. If $g=1$ and $p\nmid d$ then $\shT\simeq (\det \pi_*\shL)^{2n^2}\otimes (\pi_*\omega_\pi)^{n(dn+d-2n)}$.
If $g\geq 2$ and $p=0$ then $\shT \simeq (\det \pi_*\omega_\pi)^{-2n^2} \otimes d_\pi(\shL)^{n(n-1)} \otimes (\det \pi_*(\shL\otimes \omega_\pi))^{n(n+1)}$.
\end{prop}
\begin{proof}
 Notice that we must have $d\geq 1$, otherwise $g=0$, $d=0$ and $h=1-n\leq -1$.
 The stack $\jac_{d,g}$ is smooth and irreducible by \ref{smoothness of the universal Jacobian}. Moreover
 $\stV_{d,g,n}\arr \jac_{d,g}$ is a vector bundle of rank $nd+1-g$ and $\stU_{d,g,n}$ is the complement of the zero section in $\stV_{d,g,n}$
 by \ref{when Vdn is a vector bundle}. Notice that $\rk \stV_{d,g,n}=nd+1-g\geq 2$ using that $nd \geq \max\{2,2g-1\}$.
 Thus $\stU_{d,g,n}$ is smooth and integral and, by \ref{relations picard group smooth stacks}, we can conclude that the map $\stU_{d,g,n}\arr \jac_{d,g}$ induces an isomorphism on
 Picard groups. By \ref{relation between Bhg and Udtwo} and \ref{Bhgr not empty if d notin N} we have $\stB_{h,g,n}=\stU_{d,g,n}-\stZ_{d,g,n}\neq \emptyset$, while by \ref{relations picard group smooth stacks} and \ref{key proposition for the bad locus} the map $\gamma$ is well defined, surjective and, if $\stZ_{d,g,n}$ is integral, an isomorphism. The claim about the case $h=n=2$, $g=1$ and $p\neq 2$ follows again by \ref{relations picard group smooth stacks} and \ref{key proposition for the bad locus}.
 
 The expressions for $\shT$ in the last part
 of the statement follow by a direct computation from \ref{picard jacobian genus zero}, \ref{picard group of Xd} and
 \ref{picard universal jacobian big genus} respectively.
\end{proof}

\begin{proof}[(Proof of Theorem \ref{main theorem}, except the case of $\stB_{1,1,2}$)]
The first part of the statement follows from \ref{relation between Bhg and Udtwo} and \ref{Bhgr not empty if d notin N}.

 By construction $\shL$ is the universal invertible sheaf on $\shC$ with respect to the map $\stB_{h,g,n}\arr \jac_{d,g}$.
 We first consider the case $g=1$ and $h=n=2$. By \ref{picard group of Xd} and \ref{general prop for main theorem} the group $\Pic \stB_{2,1,2}$ is generated by $\alpha=\pi_*\shL$ and $\beta=\pi_*\omega_\pi$ with relations $8\alpha=2\beta$, $2\alpha=2\beta$
 and $12\beta=0$. Those relations are equivalent to $6\beta=0$ and $2\alpha=2\beta$, which yields $\Pic \stB_{2,1,2}\simeq \Z/6\Z \times \Z/2\Z$.
 
 In all the other cases we have that $nd>2g-2$ and that $\stZ_{d,g,n}$ is integral by \ref{key proposition for the bad locus}. In particular the map $\gamma$ defined in \ref{general prop for main theorem} is an isomorphism.
 Using \ref{general prop for main theorem} the description of $\Pic \stB_{h,g,n}$ with generators and relations follows from \ref{picard jacobian genus zero} for $g=0$, \ref{picard group of Xd} for $g=1$ and \ref{picard universal jacobian big genus} for $g\geq 2$.
 
 We now deal with the description as abstract groups. For $g=0$ the result is clear.
 
 Consider now the case  (\ref{main theorem genus one big h}), that is $g=1$ and $nd > 2$. Set $A=n(dn+d-2n)$ and notice that $2\mid A$.
 The group $H=\Pic \stB_{h,1}$ is isomorphic to $\Z^2/\langle (0,12),(2n^2,A)\rangle$. The element $(0,4)$ has order $3$ in $H$.
 A direct check shows that the map $\phi\colon H\arr \Z/3\Z$ given by $\phi(0,1)=1$ and $\phi(1,0)=x$, where $x=0$ if $3\mid n$ and $x=A/n^2$ otherwise,
 is well defined because $3\mid n$ implies $3\mid A$. Since $\phi(0,4)=1$ we obtain $H\simeq \Z/3\Z \times G$ where $G=H/\langle (0,4) \rangle$.
 We have $G=\Z^2/\langle (0,4),(2n^2,A)\rangle$. If $4\mid A$ then $G\simeq \Z/4\Z\times \Z/2n^2\Z$. So assume $A\equiv 2 \text{ mod } (4)$.
 The map $\psi\colon G\arr \Z/2\Z$ given by $\psi(u,v)=v$ is well defined. Moreover $(n^2,1)$ has order $2$ in $G$ and $\psi(n^2,1)=1$.
 We obtain $G=\Z/2\Z\times \Z^2/\langle(0,4),(n^2,1)\rangle$. It is now easy to check that the last factor is cyclic of order $4n^2$.
 
 Consider now the case (\ref{main theorem genus big}) and set $x=(1,0,0)\in \Z^3$. Then $\Pic \stB_{h,g,n}$ is isomorphic to the group  $H$  quotient of $\Z^3$
 by the relations $(-2n^2,n(n+1),n(n-1))$ and, if $g=2$, $10x$. Set $l=10$ if $g=2$ and $l=0$ otherwise. It is easy to see that
 $\langle x \rangle \simeq \Z/l\Z$. A direct computation shows that the map $\psi \colon H\arr \Z/l\Z$ given by $\psi(u,v,z)=u+v+z$
 is well defined. Since $\psi(x)=1$ we can conclude that $H \simeq \Z/l\Z \times G$, where $G=H/x \simeq \Z^2/\langle (n(n+1),n(n-1))\rangle$.
 Set $m$ for the great common divisor of $n(n+1)$ and $n(n-1)$. An easy computation shows that $m=n$ if $n$ is even and $m=2n$ is $n$ is odd.
 Let $\alpha,\beta\in \Z$ such that $\alpha n(n+1)+\beta n(n-1)=m$. Consider the map
 \[ \phi=\left( \begin{array}{cc}
\alpha & \beta \\
-\frac{n(n-1)}{m} & \frac{n(n+1)}{m}
\end{array} \right)\colon \Z^2 \arr \Z^2
\]
 By construction $\phi$ is an isomorphism because $\det \phi=1$. Moreover $\phi(n(n+1),n(n-1))=(m,0)$ and therefore
 $G\simeq \Z^2/\langle (m,0) \rangle \simeq \Z/m\Z\times \Z$ as required.
\end{proof}

In the remaining part of this section we will deal with the case of $\stB_{1,1,2}$.
As pointed out at the beginning, this case is
peculiar and needs a variation of the methods used for higher genera. Nevertheless, the steps in the computation of $\Pic \stB_{1,1,2}$ are very similar
to the ones in the computation of $\Pic \stB_{h,1,n}$, for $h\gg 0$.

In what follows we consider $g=1$ and assume that $p\nmid 6$.
We denote by $\widetilde \stM_{1,2}$ the universal curve over $\stM_{1,1}$, 
which is the moduli stack of triples $(E,\sigma_1,\sigma_2)$ where $E$ is 
a genus one curve and $\sigma_1$, $\sigma_2$ are sections. The map $\widetilde \stM_{1,2}\arr\stM_{1,1}$ is the functor that forgets the second section.

\begin{prop}\label{fundamental map genus one}
 The functor
  \[
  \begin{tikzpicture}[xscale=3.0,yscale=-0.6]
    \node (A0_0) at (0, 0) {$\widetilde \stM_{1,2}$};
    \node (A0_1) at (1, 0) {$\jac_{0,1}$};
    \node (A1_0) at (0, 1) {$(E,\sigma_1,\sigma_2)$};
    \node (A1_1) at (1, 1) {$(E,\odi E(\sigma_2-\sigma_1))$};
    \path (A0_0) edge [->]node [auto] {$\scriptstyle{}$} (A0_1);
    \path (A1_0) edge [|->,gray]node [auto] {$\scriptstyle{}$} (A1_1);
  \end{tikzpicture}
  \]
is an epimorphism in the fppf topology and it is a section of the functor
  \[
  \begin{tikzpicture}[xscale=3.0,yscale=-0.6]
    \node (A0_0) at (0, 0) {$\jac_{0,1}$};
    \node (A0_1) at (1, 0) {$\widetilde \stM_{1,2}$};
    \node (A1_0) at (0, 1) {$(E,\shQ)$};
    \node (A1_1) at (1, 1) {$(\Picsh^0_{E/S},[\odi E],[\shQ])$};
    \path (A0_0) edge [->]node [auto] {$\scriptstyle{}$} (A0_1);
    \path (A1_0) edge [|->,gray]node [auto] {$\scriptstyle{}$} (A1_1);
  \end{tikzpicture}
  \]
\end{prop}
\begin{proof}
 The second part of the statement follows from \ref{elliptic curves and pic}. For the first one let $(E,\shQ)\in \jac_{0,1}$. We can assume 
that $E$ has a section $\sigma_1$. Again by \ref{elliptic curves and pic}, $[\shQ]=[\odi E(\sigma_2-\sigma_1)]$, for some section $\sigma_2$ of $E$,
which means that
$\shQ$ and $\odi E(\sigma_2-\sigma_1)$ are fppf locally isomorphic.
\end{proof}

\begin{definit}
We define the group functor $G_0$ over $\widetilde \stM_{1,2}$ as the group $G$ obtained as in
\ref{torsors over a stack} with respects to the maps defined in \ref{fundamental map genus one}.
\end{definit}

\begin{prop}\label{properties of Gzero}
An element of $G_0(S\arrdi{(E,\sigma_1,\sigma_2)}\widetilde \stM_{1,2})$ is a pair $(f,\lambda)$ 
where $f\colon (E,\sigma_1)\arr (E,\sigma_1)$ is a translation and 
$\lambda\colon \odi E(\sigma_2-\sigma_1)\arr \odi E(f(\sigma_2)-f(\sigma_1))$ is an isomorphism. Moreover we have an exact sequence
  \[
  \begin{tikzpicture}[xscale=1.5,yscale=-0.6]
    \node (A0_2) at (2, 0) {$(f,\lambda)$};
    \node (A0_3) at (3, 0) {$f(\sigma_1)$};
    \node (A1_0) at (0, 1) {$0$};
    \node (A1_1) at (1, 1) {$\Gm$};
    \node (A1_2) at (2, 1) {$G_0$};
    \node (A1_3) at (3, 1) {$\E$};
    \node (A1_4) at (4, 1) {$0$};
    \node (A2_1) at (1, 2) {$\mu$};
    \node (A2_2) at (2, 2) {$(\id,\mu)$};
    \path (A2_1) edge [|->,gray]node [auto] {$\scriptstyle{}$} (A2_2);
    \path (A1_0) edge [->]node [auto] {$\scriptstyle{}$} (A1_1);
    \path (A1_1) edge [->]node [auto] {$\scriptstyle{}$} (A1_2);
    \path (A1_2) edge [->]node [auto] {$\scriptstyle{}$} (A1_3);
    \path (A0_2) edge [|->,gray]node [auto] {$\scriptstyle{}$} (A0_3);
    \path (A1_3) edge [->]node [auto] {$\scriptstyle{}$} (A1_4);
  \end{tikzpicture}
  \]
in the Zariski topology of $\Sch/\widetilde \stM_{1,2}$, 
where $\E\arr \widetilde \stM_{1,2}$ is the universal curve. In particular $G_0$ is smooth over $\widetilde \stM_{1,2}$.
\end{prop}
\begin{proof}
 By definition, an element of $G_0(S\arrdi{(E,\sigma_1,\sigma_2)}\widetilde \stM_{1,2})$ is a pair $(f,\lambda)$ where $f\colon E\arr E$ is an isomorphism
such that $f_*\colon \Picsh^0_{E/S}\arr \Picsh^0_{E/S}$ is the identity and 
$\lambda\colon \odi E(\sigma_2-\sigma_1)\arr \odi E(f(\sigma_2)-f(\sigma_1))$ is an isomorphism. 
As in the proof of \ref{properties of Gd}, $f_*=\id$ means that $f$ is a translation.
In particular the sequence in the statement is well defined and, since $\pi_*\mathbb{G}_{m,E}\simeq \mathbb{G}_m$, it is exact in the first two terms.

It remains to prove that $G_0\arr \E$ is a Zariski epimorphism. This will also imply that $G_0$ is locally a product
of $\Gm$ and $\E$ and therefore smooth. Let $(E,\sigma_1,\sigma_2,\delta) \in \E$ and $t$ be the translation by $\delta$,
so that $t(\sigma_1)=\delta$.
Since $t$ is a translation
we have
\[
[\odi E(\sigma_2-\sigma_1)]= [\odi E(t(\sigma_2)-t(\sigma_1))]\text{ in }\Picsh^0_{E/S}
\]
which means that the sheaves differ from an invertible sheaf coming from the base thanks to
\ref{cbs for elliptic curves}. So Zariski locally we get an isomorphism
$\lambda\colon \odi E(\sigma_2-\sigma_1)\arr\odi E(t(\sigma_2)-t(\sigma_1))$ and therefore $(t,\lambda)\in G_0(E,\sigma_1.\sigma_2))$ is mapped to
$(E,\sigma_1,\sigma_2,\delta)$.
\end{proof}

\begin{lemma}\label{gm-torsors and picard}
 Let $\stX$ be a smooth algebraic stack, $\stY\arrdi p\stX$ be a $\Gm$-torsor and $\shL$ be the invertible sheaf over $\stX$ corresponding
 to it. Then we have an exact sequence
 \[
 \Z\arrdi \shL \Pic \stX\arrdi{p^*} \Pic \stY\arr 0
 \]
\end{lemma}
\begin{proof}
 The stack $\stY$ and the line bundle $\stY'\arr\stX$ corresponding to $\shL$ can be seen as the relative sheaves 
 \begin{displaymath}
  \Isosh_\stX(\odi\stX,\shL) \text{ and } \Homsh_\stX(\odi\stX,\shL)
 \end{displaymath}
on $\Sch/\stX$ respectively. Denote by $q\colon \stY'\arr\stX$ the structural morphism.
  The stack $\stY$ is the open substack of $\stY'$ whose complement $\stZ$ is the zero section of $\stY'\arr\stX$.
 The stack $\stZ$ is integral since $\stX$ is so and it is the zero locus of the universal section of $q^*\shL$. The result then follows from
 \ref{relations picard group smooth stacks}.
\end{proof}

In what follows we denote by $\stF$ the complement of the zero section $\stM_{1,1}\arr \widetilde \stM_{1,2}$ in $\widetilde \stM_{1,2}[2]$.

\begin{prop}\label{picard of F is picard of Bone}
 The composition $\stB_{1,1,2}\arr \jac_{0,1} \arr \widetilde{\stM}_{1,2}$ has image in $\stF$. The induced map $\stB_{1,1,2}\arr \stF$
 yields an isomorphism on Picard groups and factors as a $\Gm$-torsor $\stB_{1,1,2}\arr \Bi_\stF G_0$ followed by the projection
 $\Bi_\stF G_0 \arr \stF$.
\end{prop}
\begin{proof}
 Set $\stX=\Bi_\stF G_0$. By \ref{torsors over a stack} and \ref{fundamental map genus one}, we see that $\jac_{0,1}\simeq \Bi_{\widetilde \stM_{1,2}} G_0$.
 In particular $\stX$ can be seen as the closed substack of $\jac_{0,1}$
 of pairs $(E,\shQ)$ such that $(\Picsh^0_E,[\odi E],[\shQ])\in\stF$. By \ref{relation between Bhg and Udtwo}
 we see that the forgetful map $\stB_{1,1,2}\arr\stX$ is a
 $\Gm$-torsor corresponding to the invertible sheaf $\pi_*(\shL^2)$, where $\pi\colon\E\arr\stX$ is the universal curve and $\shL$ is the universal invertible sheaf
 over it. 
 Notice that $\stX$ is smooth because it is an open substack of $\jac_{0,1}$, which is smooth thanks to \ref{smoothness of the universal Jacobian}.
 In particular from \ref{gm-torsors and picard} the pull-back of $\stB_{1,1,2}\arr\stX$ induces an isomorphism
 \[
 \Pic\stB_{1,1,2} \simeq \Pic \stX/\langle \pi_*(\shL^2)\rangle
 \]
 Moreover from \ref{picard of BG}, we have $\Pic\stX=\Pic\stF \oplus \duale G_0$. We are going to show that $\duale G_0\simeq \Z$ and that the
 component of $\pi_*(\shL^2)$ in $\Pic\stX$ with respect to $\duale G_0$ generates this last group. This will imply that the composition of pull-backs
 $\Pic\stF\arr\Pic\stX\arr\Pic\stB_{1,1,2}$ is an isomorphism.
 
 Taking into account \ref{properties of Gzero}, the inclusion $\Gm\arr G_0$ yields a map $\alpha\colon \duale G_0\arr \duale{\mathbb G}_m\simeq \Z$ whose kernel is the
 group of characters of the universal curve $\widetilde \E$ of $\stF$. If $\phi\colon \widetilde \E\arr\Gm$ is such a homomorphism, $\phi^{-1}(1)$
 is a closed substack of $\widetilde \E$
 and, by checking on the geometric fibers, we see that they are topologically equal. Since 
 $\widetilde \E$ is reduced we can conclude that $\phi$ is trivial and therefore
 that $\alpha$ is injective. As in the proof of \ref{computation of det pi star L to the n}, considering the functor
   \[
  \begin{tikzpicture}[xscale=4.5,yscale=-0.6]
    \node (A0_0) at (0, 0) {$\Bi_\stF \Gm$};
    \node (A0_1) at (1, 0) {$\stX$};
    \node (A1_0) at (0, 1) {$(E\arrdi \pi S,\sigma_1,\sigma_2,\shQ)$};
    \node (A1_1) at (1, 1) {$(E,\odi E(\sigma_2 - \sigma_1)\otimes \pi^*\shQ)$};
    \path (A0_0) edge [->]node [auto] {$\scriptstyle{}$} (A0_1);
    \path (A1_0) edge [|->,gray]node [auto] {$\scriptstyle{}$} (A1_1);
  \end{tikzpicture}
  \]
and the expression
\[
\pi_*[(\odi E(\sigma_2-\sigma_1)\otimes \pi^*\shQ)^2]\simeq \pi_*(\odi E(2\sigma_2-2\sigma_1))\otimes \shQ^2
\]
we see that $\pi_*(\shL^2)$ is sent to $2$ by the map $\Pic \stX\simeq \Pic \shF \oplus \duale G_0 \arrdi{\id\oplus \alpha }\Pic \shF \oplus \duale{\mathbb G}_m\arr \duale{\mathbb G}_m=\Z$.
In particular $2\Z\subseteq \Imm\alpha$ and we need to prove that those groups are equal, or, equivalently, that $\alpha$
 is not an isomorphism. Assume by contradiction that $\alpha$ is an isomorphism.
 This exactly means that the map $\Gm\arr G_0$ has a section. Thus also the map $G_0\arr\widetilde\E$ has a section.
 Since this last map is a $\Gm$-torsor, we can rephrase this saying that the invertible sheaf over 
$\widetilde \E$ corresponding to 
$G_0$ is trivial. We are going to compute this sheaf and prove that it is not trivial. Given 
$(E\arrdi \pi S,\sigma_1,\sigma_2,\sigma_3)\in\widetilde \E$ and denoted by
$t\colon (E,\sigma_1)\arr (E,\sigma_1)$ the translation by $\sigma_3$, so that $t(\sigma_1)=\sigma_3$, the invertible sheaf $\shK$ over $\widetilde \E$ corresponding to
$G_0$ is given by the following calculation
\begin{align*}
 \Isosh_E(\odi E(\sigma_2 - \sigma_1),\odi E(t(\sigma_2) - t(\sigma_1)) & \simeq \Isosh_E(\odi E(\sigma_2+\sigma_3-\sigma_1),\odi E(t(\sigma_2))) \\
 & \simeq \Isosh_S(\pi_*\odi E(\sigma_2+\sigma_3-\sigma_1),\odi S) \simeq \duale{\pi_*\odi E(\sigma_2+\sigma_3-\sigma_1)}
\end{align*}
where we have used that $\pi_*\odi E(t(\sigma_2))\simeq \odi S$ by \ref{degree one sheaves and sections}.
Using \ref{computation of line bundles on the universal elliptic curve} twice we also have
\[
\duale{\pi_*\odi E(\sigma_2+\sigma_3-\sigma_1)} \simeq \sigma_1^*\odi E(\sigma_2)\otimes\sigma_1^*\odi E(\sigma_3)\otimes \sigma_3^*\odi E(-\sigma_2)\otimes \pi_*\omega_\pi
\]
Given an elliptic curve $E$ over an algebraically closed field with origin $p_1$ and $p_2\in E[2]-\{p_1\}$, we consider the object
$\chi=(E\times E \arrdi{\pr_2} E,\sigma_1,\sigma_2,\Delta)\in \widetilde \E(E)$ where $\sigma_i=p_i\times \id_E\colon E\arr E\times E$ for $i=1,2$ and 
$\Delta\colon E\arr E\times E$ is the diagonal. Using isomorphism above, the pull-back of $\shK$ to $E$ is given by
\[
\sigma_1^*\odi {E\times E}(\sigma_2)\otimes\sigma_1^*\odi{E\times E}(\Delta)\otimes \Delta^*\odi{E\times E}(-\sigma_2)\otimes \pr_{2*}\omega_{\pr_2}
\simeq \odi E(p_1-p_2)
\]
which is not trivial.
\end{proof}

\begin{prop}\label{picard of F}
 We have $\Pic \stF\simeq \Z/4\Z$, generated by the invertible sheaf $\pi_*\omega_\pi$,
 where $\pi\colon \E\arr \shF$ is the universal curve over $\shF$.
\end{prop}
\begin{proof}
 Let $k$ be the base field and set $U=\Spec R$,
 where $R=k[a,b]_\Delta$ with $\Delta=4a^3+27b^2$ and $\rho\colon \stF\arr\stM_{1,1}$ for the structure map,
 which is an \'etale degree $3$ cover. Since $\car k\nmid 6$, the map $U\arr\stM_{1,1}$ given by the general Weierstrass curve
 \[
 W=\Proj (R[x,y,z]/(f))\arr U\text{ where } f=y^2z-x^3-axz^2-bz^3
 \]
 is a $\Gm$-torsor corresponding
 to a generator $\shK$ of $\Pic \stM_{1,1}$, either $p_*\omega_p$ or its dual, where $p\colon \widetilde \E\arr\stM_{1,1}$ is the universal curve. In particular
 the base change $V=U\times_{\stM_{1,1}} \shF\arr\shF$ is the $\Gm$-torsor corresponding to the pull-back $\rho^*\shK$ and it coincides with
 $W[2]$ minus the zero section. By \cite[Group law algorithm 2.3]{Silverman1986}, $2$-torsion points are obtained modding out by $y$
 and therefore we get
 \[
 V\simeq \Spec k[a,b,x]_\Delta/(x^3+ax+b)
 \]
 In particular $V$ is an open subscheme of $\A^2=\Spec k[a,x]$ and therefore $\Pic V=0$. By \ref{gm-torsors and picard} we can conclude that $\Pic\stF$
 is generated by $\omega=\pi_*\omega_\pi$. This is because, if $\omega'=p_*\omega_p$, then $\rho^*\omega'\simeq \omega$.
 
 We want to prove that the order $r$ of $\omega$ in $\Pic \shF$ is exactly $4$. Since $\omega'$ has order $12$ in $\Pic\stM_{1,1}$
 by \ref{Mumford's theorem}, we have
 \[
 \omega^r\simeq \rho^*\omega'^r\simeq \odi \shF \then \rho_*\odi \shF\simeq \omega'^r\otimes \rho_*\odi\shF 
 \stackrel{\det}{\then} \omega'^{3r}\simeq\odi{\stM_{1,1}}\then 12\mid 3r\then 4\mid r
 \]
 Consider now the invertible sheaf $\shT$ on $\shF$ given by
 \[
 \shT\colon (E\arrdi q S,\sigma_1,\sigma_2)\longmapsto \sigma_2^*\odi E(\sigma_2-\sigma_1) \otimes \sigma_1^*\odi E(\sigma_1-\sigma_2)
 \]
 Since $\odi E(\sigma_2-\sigma_1)^2 \simeq q^*q_* (\odi E(\sigma_2-\sigma_1)^2) $ 
 by definition of $\shF$ and \ref{cbs for elliptic curves}, we see that $\shT^2\simeq \odi \shF$. On the other hand,
 since $\sigma_1$ and $\sigma_2$ are disjoint, we have $\sigma_1^*\odi E(\sigma_2)\simeq\sigma_2^*\odi E(\sigma_1)\simeq \odi E$ and therefore
 $\shT\simeq \omega^2$ thanks to \ref{computation of line bundles on the universal elliptic curve}. In conclusion $\odi\stF\simeq \shT^2\simeq \omega^4$ and therefore
 $r=4$.
\end{proof}

\begin{proof}[(Proof of Theorem \ref{main theorem}, the case of $\stB_{1,1,2}$)]
By \ref{relation between Bhg and Udtwo} $\shL$ is a degree $0$ invertible sheaf on $\shC$ which is never trivial on the geometric fibers of $\shC\arr \stB_{1,1,2}$. By Grauert we can conclude that $\pi_*\shL=0$, so that  $\det \pi_*\shL$ is trivial. The result then follows from \ref{picard of F is picard of Bone} and \ref{picard of F}.
\end{proof}

\bibliography{biblio}{}

\begin{thebibliography}{FvdG04}

\bibitem[AV04]{Arsie2004}
Alessandro Arsie and Angelo Vistoli.
\newblock {Stacks of cyclic covers of projective spaces}.
\newblock {\em Compositio Mathematica}, 140(3):647--666, May 2004.

\bibitem[BV12]{Bolognesi2009}
Michele Bolognesi and Angelo Vistoli.
\newblock {Stacks of trigonal curves}.
\newblock {\em Transactions of the American Mathematical Society},
  364:3365--3393, 2012.

\bibitem[FO10]{Fulton2010}
William Fulton and Martin Olsson.
\newblock {The Picard group of $M_{1,1}$}.
\newblock {\em Algebra and Number Theory}, 4:87--104, April 2010.

\bibitem[FvdG04]{Faber2004}
Carel Faber and Gerard van~der Geer.
\newblock {Complete subvarieties of moduli spaces and the Prym map}.
\newblock {\em Journal f\"{u}r die reine und angewandte Mathematik (Crelles
  Journal)}, 2004(573), January 2004.

\bibitem[Gro66]{EGAIV-3}
Alexander Grothendieck.
\newblock {\em {EGAIV-3 - \'{E}tude locale des sch\'{e}mas et des morphismes de
  sch\'{e}mas (Troisi\'{e}m partie) - \'{E}l\'{e}ments de g\'{e}om\'{e}trie
  alg\'{e}brique (r\'{e}dig\'{e}s avec la collaboration de Jean
  Dieudonn\'{e})}}.
\newblock Inst. Hautes \'{E}tudes Sci. Publ. Math. 28, 1966.

\bibitem[Gro67]{EGAIV-4}
Alexander Grothendieck.
\newblock {\em {EGAIV-4 - \'{E}tude locale des sch\'{e}mas et des morphismes de
  sch\'{e}mas (Quatri\`{e}me partie) - \'{E}l\'{e}ments de g\'{e}om\'{e}trie
  alg\'{e}brique (r\'{e}dig\'{e}s avec la collaboration de Jean
  Dieudonn\'{e})}}.
\newblock Inst. Hautes \'{E}tudes Sci. Publ. Math. 32, 1967.

\bibitem[Kle80]{Kleiman1980}
Steven~L. Kleiman.
\newblock {Relative Duality for Quasi-Coherent Sheaves}.
\newblock {\em Compositio Mathematica}, 41:39 -- 60, 1980.

\bibitem[MBL99]{Laumon1999}
Laurent Moret-Bailly and Gerard Laumon.
\newblock {\em {Champs alg\'{e}briques}}.
\newblock Springer, first edition, 1999.

\bibitem[Mil08]{Milne2008}
James~S. Milne.
\newblock {\em {Abelian Varieties}}.
\newblock Online lecture notes, 2008.

\bibitem[Mum63]{Mumford1963}
David Mumford.
\newblock {Picard groups of moduli problems}.
\newblock {\em Arithmetical Algebraic Geometry}, pages 33--81, 1963.

\bibitem[Mum66]{Mumford1966}
David Mumford.
\newblock {On the equations defining abelian varieties. I}.
\newblock {\em Inventiones Mathematicae}, 1(4):287--354, December 1966.

\bibitem[MV14]{melo-viviani}
Margarida Melo and Filippo Viviani.
\newblock {The Picard group of the compactified universal Jacobian}.
\newblock {\em Documenta Mathematica}, 19:457--506, July 2014.

\bibitem[Pag13]{Pagani2013}
Nicola Pagani.
\newblock {Moduli of abelian covers of elliptic curves}.
\newblock {\em arxiv id: 1303.2991}, page~24, March 2013.

\bibitem[Sil86]{Silverman1986}
Joseph~H. Silverman.
\newblock {\em {The Arithmetic of Elliptic Curves}}.
\newblock 1986.

\bibitem[SP014]{SP014}
{Stacks Project}, 2014.

\end{thebibliography}
\bibliographystyle{alpha}

\end{document}